\documentclass[11pt]{article}

\usepackage{amsmath, amsfonts,amssymb, amsthm}
\usepackage{mathtools}
\usepackage{multirow}
\usepackage{color}
\usepackage[table]{xcolor}
\usepackage{epsf}
\usepackage{tikz}
\usepackage{hyperref}
\hypersetup{
    colorlinks = true,
    linkcolor = {black},
    citecolor = {black}
}

\usepackage[left=1.17in, right=1.17in, top=1.11in, bottom=1.19in]{geometry}

\renewenvironment{thebibliography}[1]{
  \begin{oldthebibliography}{#1}
    \setlength{\itemsep}{.5em}
    \setlength{\parskip}{0em}
}
{
  \end{oldthebibliography}
}

\usepackage{graphicx}
\usepackage{caption}
\usepackage{subcaption}

\usepackage{enumitem}
\setlist{itemsep=0pt, topsep=0pt}

\newcommand{\floor}[1]{\lfloor#1\rfloor}
\newcommand{\ceiling}[1]{\lceil#1\rceil}

\newcommand{\ex}{\mathrm{ex}}

\newcommand{\tbf}[1]{\textbf{#1}}

\newtheorem{theorem}{Theorem}[section]
\newtheorem{corollary}[theorem]{Corollary}
\newtheorem{lemma}[theorem]{Lemma}

\newtheorem{claim}[theorem]{Claim}

\newtheorem{observation}[theorem]{Observation}

\newtheorem{example}[theorem]{Example}
\newtheorem{problem}[theorem]{Problem}

\newtheorem{remark}[theorem]{Remark}

\newenvironment{proofclaim}[1][Proof of claim]{\begin{proof}[#1]}{\end{proof}}

\newcommand{\ep}{\epsilon}

\usepackage{color}

\title{On the Ramsey numbers of fans and stars}
\author{Louis DeBiasio\thanks{Department of Mathematics, Miami University, Oxford, OH. \texttt{debiasld@miamioh.edu}. Research supported in part by NSF grant DMS-1954170 and AMS-Simons Research Enhancement Grants for PUI Faculty GR001015.}, Tucker Wimbish\thanks{Department of Mathematics, Miami University, Oxford, OH. \texttt{wimbistj@miamioh.edu}.}}
\date{\today}

\begin{document}

\maketitle

\begin{abstract}
Let $F_n$ be the graph on $2n+1$ vertices consisting of $n$ triangles meeting at a single vertex.  After a number of improvements over the years, it is currently known that the Ramsey number of $F_n$ is between $4.5n-5$ (Chen, Yu, Zhao \cite{CYZ}) and $(5+\frac{1}{6})n+O(1)$ (Dvo{\v{r}}{\'a}k and Metrebian \cite{DM}).  We improve both of these bounds as follows
$$4.732n\approx (3+\sqrt{3})n-8< R(F_n)\leq (5+o(1))n.$$

Additionally, as it relates to the lower bound on $R(F_n)$ (and for which nothing was known when $n< m< n(n-1)$), we determine the Ramsey numbers of stars vs.~fans, within a constant, as follows
$$R(K_{1,m}, F_{n})= 
    \begin{cases} 
     m+2n-\frac{1+(-1)^{m}}{2}, & m\leq n \\
     \frac{3m+\sqrt{m^2+8n^2}}{2}+\Theta(1), & m>n
   \end{cases}.
$$
In particular, we have $R(K_{1,2n}, F_n)=(3+\sqrt{3})n+\Theta(1)$.
\end{abstract}

\section{Introduction}

\subsection{Ramsey}

The Ramsey number of a graph $H$, denoted $R(H)$ is the smallest $N$ such that in every $2$-coloring of the edges of $K_N$ there is a monochromatic copy of $H$.  Given graphs $H_1$ and $H_2$ we write $R(H_1, H_2)$ to be the smallest $N$ such that in every $2$-coloring of the edges of $K_N$ there exists a copy of $H_1$ in color $1$ or a copy of $H_2$ in color $2$.  

The $n$-fan, $F_n$, (also known as the ``$n$-friendship graph''\footnote{Erd\H{o}s, R\'enyi, and S\'os \cite{ERS} proved that the only graphs with the property that every pair of vertices have exactly one common neighbor are the graphs $F_n$.  This result is known as ``the friendship theorem."}) is the graph on $2n+1$ vertices consisting of $n$ triangles meeting at a single vertex; in other words, a vertex joined to a matching of size $n$. The book with $n$ pages, $B_n$, is the graph on $n+2$ vertices consisting of $n$ triangles meeting at a single edge.  Given a graph $H$ and a positive integer $n$, we write $nH$ to be the vertex disjoint union of $n$ copies of $H$.

Burr, Erd\H{o}s, and Spencer \cite{BES} proved that $R(nK_3)=5n$ for all $n\geq 2$, and Rousseau and Sheehan \cite{RS1} proved that $R(B_n)\leq 4n+2$ for all $n\geq 1$ and that this bound is tight for infinitely many $n$ (in particular whenever a self-complementary strongly-regular graph on $4n+1$ vertices exists).  However, as evidenced by the number of partial results and the large gap between the best known upper and lower bounds, the problem of determining $R(F_n)$ appears to be much harder.  

Li and Rousseau \cite{LiR} first proved that $4n+1\leq R(F_n)\leq 8n-2$ where the lower bound comes from the fact that $F_n$ has $2n+1$ vertices and chromatic number 3 (take two disjoint red cliques on $2n$ vertices with all blue edges between them).  13 years later, Lin and Li \cite{LL} improved the upper bound by showing that $R(F_n)\leq 2\cdot R(nK_2, F_n)=6n$.  Another 13 years later, Chen, Yu, and Zhao \cite{CYZ}, improved both the upper and lower bound by showing $4.5n-5\leq R(F_n)\leq 5.5n+O(1)$.  Note that both of the improvements in \cite{CYZ} required going beyond the ``basic'' bounds of $(|V(F_n)|-1)(\chi(F_n)-1)+1\leq R(F_n)\leq 2\cdot R(nK_2, F_n)$.  Finally, Dvo{\v{r}}{\'a}k and Metrebian \cite{DM} refined the methods of Chen, Yu, and Zhao to improve the upper bound to $R(F_n)\leq (5+\frac{1}{6})n+O(1)$.

We prove the following.
\begin{theorem}\label{thm:main}
For all $\ep>0$ and all $n\geq \frac{384}{\ep^2}$, $$(3+\sqrt{3})n-8< R(F_n)\leq (5+\ep)n.$$
\end{theorem}

For the upper bound, our improvement is based on the following idea:  By the result of Lin and Li mentioned above, we would be done if any vertex has red degree or blue degree at least $3n$, so every vertex has red degree and blue degree less than $3n$ and consequently greater than $2n$.  Given any vertex $v$, if there were a red matching covering $2n$ vertices in $N_R(v)$, or a blue matching covering $2n$ vertices in $N_B(v)$, we would be done; so suppose not.  Now by applying the Edmonds-Gallai theorem inside \emph{both} $N_R(v)$ and $N_B(v)$ we obtain fairly precise structural information about the coloring.  In particular, we have a large blue complete multipartite graph inside $N_R(v)$ and a large red complete multipartite graph inside $N_B(v)$.  Now the proof proceeds by analyzing the sizes of these monochromatic complete multipartite graphs as well as the edges incident with them (in particular, between them).  We will discuss this idea in more detail later as well as comparing our idea to the previous proofs of Chen, Yu, and Zhao \cite{CYZ} and Dvo{\v{r}}{\'a}k and Metrebian \cite{DM}.

Our improvement for the lower bound begins with the following observation: note that $K_{1,2n}\subseteq F_n$ and thus $R(F_n, K_{1,2n})\leq R(F_n, F_n)=R(F_n)$.  Without explicitly phrasing it this way, Chen, Yu, and Zhao \cite{CYZ} proved that
\begin{equation}\label{eq:fanstar}
4.5n-5\leq R(K_{1,2n}, F_n)\leq R(F_n).
\end{equation}
So any improvement in the lower bound on $R(K_{1,2n}, F_n)$ would improve the lower bound on $R(F_n)$.  Alternatively, proving an upper bound of $4.5n+O(1)$ on $R(K_{1,2n}, F_n)$ would show that the lower bound cannot be improved by considering $R(K_{1,2n}, F_n)$ alone. However, to the best of our knowledge, an upper bound on $R(K_{1,2n}, F_n)$ has never been studied.  In an attempt to provide such an upper bound, we surprisingly showed that $$R(K_{1,2n}, F_n)=(3+\sqrt{3})n+\Theta(1)\approx 4.732n.$$
In fact, we prove a more general result which asymptotically determines the value of $R(K_{1,m}, F_{n})$ for all $n$ and $m$.  

\begin{theorem}\label{thm:fanstar}For all integers $n$ and $m$ with $m>n\geq 2$,  
we have $$\frac{3m+\sqrt{m^2+8n^2}}{2}-8<R(K_{1,m}, F_{n})< \frac{3m+\sqrt{m^2+8n^2}}{2}+1.$$  In particular, $(3+\sqrt{3})n-8< R(K_{1,2n}, F_n)< (3+\sqrt{3})n+1.$
\end{theorem}

It is known \cite{Has} (see \cite[Theorem 1]{LiSc}) that $1\leq m\leq n$, then $R(K_{1,m}, F_{n})= m+2n-1$ if $m$ is even and $R(K_{1,m}, F_{n})= m+2n$ if $m$ is odd
    \footnote{For the upper bound, if there is no red $K_{1,m}$, then some vertex $v$ will have blue degree at least $2n$.  Indeed, if $m$ is odd, this is clear. If $m$ is even, then $N=m+2n-1$ is odd and $m-1$ is odd, so some vertex must have red degree at most $m-2$ and consequently blue degree at least $2n$. Furthermore, every vertex in $N_B(v)$ will have at least $n$ blue neighbors in $N_B(v)$ giving a matching of size $n$ (by a well-known folklore result).  For the lower bound, an $(m-1)$-regular graph (so the complement has maximum degree at most $2n-1$) on $N=m+2n-\frac{1+(-1)^m}{2}-1$ exists (since either $N$ is even or $m-1$ is even).}.
Also when $m\geq n(n-1)$, Zhang, Broersma, and Chen \cite{ZBC3} proved that $R(K_{1,m}, F_n)=2m+1$.  Note that when $m\geq n(n-1)$, this matches Theorem \ref{thm:fanstar} within a constant.

So by combining Theorem \ref{thm:fanstar} with the previously known cases we have
\begin{equation}\label{eq:starfan}
R(K_{1,m}, F_{n})= 
    \begin{cases} 
     m+2n-\frac{1+(-1)^{m}}{2} & m\leq n \\
     \frac{3m+\sqrt{m^2+8n^2}}{2}+\Theta(1) & n< m< n(n-1)\\
     2m+1 & m\geq n(n-1)
   \end{cases}.
\end{equation}

\subsection{Dirac}

As a standalone result, it may be more natural to state Theorem \ref{thm:fanstar} as a result about graphs with sufficiently large minimum degree.  The following observation gives the link between the two formulations.

\begin{observation}
Let $m$ be a positive integer and let $H$ be a graph.  $R(K_{1,m}, H)\leq N$ if and only if every graph $G$ on $N$ vertices with $\delta(G)\geq N-m$ contains a copy of $H$.
\end{observation}

So we may rephrase Theorem \ref{thm:fanstar} as follows.  

\begin{theorem}\label{thm:dirac_fan_approx}
Let $n$ and $k$ be positive integers with $2k+1\leq n$ and define $\alpha$ so that $k=\alpha n$ (note that $\alpha$ may depend on $n$ if $k=o(n)$).  If $G$ is a graph on $n$ vertices with 
$$\delta(G)\geq 
\begin{cases}
\frac{n+1}{2} & k<\sqrt{n}\\
(\frac{1+\sqrt{1+16\alpha^2}}{4})n +\Theta(1) & \sqrt{n}\leq k< \frac{n}{3}\\
2k & \frac{n}{3}\leq k< \frac{n}{2},
\end{cases}
$$
then $F_{k}\subseteq G$.  Furthermore, this is best possible (asymptotically so in the case when $\sqrt{n}\leq k<\frac{n}{3}$).
\end{theorem}

\subsection{Tur\'an}

Given a graph $H$, let $\ex(n,H)$ be the maximum number of edges in an $H$-free $n$-vertex graph.  The quantity $\ex(n,H)$ is typically referred to as the Tur\'an number of $H$.  As is well known, the Tur\'an number of a graph $H$ is related to its Ramsey number in the sense that if we have a 2-coloring of $K_N$ in which one of the color classes has more than $\ex(N,H)$ edges, we would get a monochromatic copy of $H$.  This is potentially useful if $\frac{\ex(n,H)}{\binom{n}{2}}$ is not much bigger than $1/2$.

The Tur\'an number of $F_k$ was determined by Erd\H{o}s, F\"uredi, Gould, and Gunderson \cite{EFGG} for all $1\leq k\leq \sqrt{\frac{n}{50}}$.  
They proved
\[
\ex(n,F_k)= 
    \begin{cases}
     \floor{\frac{n^2}{4}}+k^2-k & k \text{ is odd} \\
      \floor{\frac{n^2}{4}}+k^2-\frac{3k}{2} & k \text{ is even},
   \end{cases}
\]
where the lower bound roughly comes from a complete balanced bipartite graph on $n$ vertices with two disjoint cliques of order $k-1$ inside one of the parts.  

They conjectured that their result should hold for all $1\leq k\leq \frac{n}{4}$.  Inspired by Theorem \ref{thm:dirac_fan_approx}, we raise the following asymptotic version of the problem which also extends their conjecture to the case when $k>\frac{n}{4}$.  

\begin{problem}\label{prob:turan}
Let $n$ and $k$ be integers such that $k=\alpha n$ for some $0<\alpha<\frac{1}{2}$.  Is it true that
$$
\ex(n,F_{\alpha n})=
\begin{cases}
\left(\frac{1}{2}+2\alpha^2+o(1)\right)\frac{n^2}{2} & 0<\alpha\leq \frac{1}{4}\\
(2\alpha (2-3\alpha)+o(1))\frac{n^2}{2} &\frac{1}{4}<\alpha\leq \frac{1}{3}\\
(2\alpha+o(1))\frac{n^2}{2} & \frac{1}{3}<\alpha< \frac{1}{2}
\end{cases}?
$$
\end{problem}

For the lower bounds, if $\frac{1}{4}<\alpha\leq \frac{1}{3}$, then let $H_{n,k}$ be the complete tripartite graph with parts of size $k-1=\alpha n-1$, $k-1=\alpha n-1$, and $(1-2\alpha)n-2$ respectively.  If $\frac{1}{3}<\alpha<\frac{1}{2}$, and $n$ is even, then let $H_{n,k}$ be a $2k-1=2\alpha n-1$ regular graph; otherwise let $H_{n,k}$ be the graph with the most edges in which every vertex has degree $2k-2=2\alpha n-2$ or $2k-1=2\alpha n-1$.

It is also worth noting that the extremal examples when $\alpha<1/3$ are not close to regular which indicates why the bound on the number of edges for the Dirac version of the problem is smaller than the conjectured bound on the number of edges in the Tur\'an version of the problem.  
Also it is interesting to note that Zhang, Broersma, and Chen \cite{ZBC3} proved that if $G$ is an $n$-vertex graph with $\delta(G)\geq \frac{n}{2}+1$ and $k<\sqrt{n}$, then $F_k\subseteq G$ and for the Tur\'an version of the problem, the result \cite{EFGG} holds when $k=O(\sqrt{n})$.


\subsection{Overview}

In Section \ref{sec:starfan} we prove Theorem \ref{thm:fanstar} regarding the Ramsey numbers of stars vs.~fans (the lower bound of which provides the lower bound in Theorem \ref{thm:main}).  In Section \ref{sec:fan_upper} we prove the upper bound in Theorem \ref{thm:main}.  In Section \ref{sec:conclusion} we discuss some additional problems.  In the appendix, we include a proof for the lower bound on $R(K_{1,m}, F_n)$ in the specific case where $m=2n$.

\subsection{Notation}

Let $G$ be a graph and let $X,Y\subseteq V(G)$ with $X\cap Y=\emptyset$.  We write $G[X]$ for the graph on $X$ induced by edges with both endpoints in $X$ and $G[X,Y]$ for the bipartite graph on $X\cup Y$ induced by edges with one endpoint in $X$ and the other in $Y$.  Given a graph $G$ and a 2-coloring of the edges of $G$, we write $G_i$ for the graph on $V(G)$ induced by the edges of color $i$.  We write $N_i(v)$ for the neighborhood of $v$ in $G_i$ and $d_i(v)=|N_i(v)|$ for the degree of $v$ in $G_i$.  
Often we use $B$ and $R$ for the subscripts (to mean ``blue'' and ``red'' respectively) in place of 1 and 2.

Given a graph $G$, we write $\nu(G)$ for the number of edges in a maximum matching of $G$.

\section{Stars vs.~fans}\label{sec:starfan}

In this section we prove Theorem \ref{thm:fanstar} which says that for all $m>n$, 
$$\frac{3m+\sqrt{m^2+8n^2}}{2}-8<R(K_{1,m}, F_{n})< \frac{3m+\sqrt{m^2+8n^2}}{2}+1.$$

\subsection{Lower bound}\label{subsec:starfan_lower}

To construct our lower bound example, we first need a result about the realization of bipartite graphs satisfying certain degree conditions.  To prove that result, we will use the Gale-Ryser theorem.  We say that a pair of sequences $(x_1, \dots, x_a)$ and $(y_1, \dots, y_b)$ of non-negative integers is \emph{bigraphic} if there exists a bipartite graph with parts $X$ and $Y$ such that $(x_1, \dots, x_a)$ is the degree sequence of the vertices in $X$ and $(y_1, \dots, y_b)$ is the degree sequence of the vertices in $Y$.

\begin{theorem}[Gale \cite{Gale}, Ryser \cite{Ry}]\label{thm:GR}
A pair of sequences $(x_1, \dots, x_a)$ and $(y_1, \dots, y_b)$ of non-negative integers with $x_1\geq x_2\geq \dots \geq x_a$ is bigraphic if and only if $\sum_{i=1}^ax_i=\sum_{i=1}^by_i$ and for all $1\leq k\leq a$, $\sum_{i=1}^kx_i\leq \sum_{i=1}^b\min\{y_i,k\}$.
\end{theorem}

\begin{lemma}\label{lem:GR}
Let $a,b$ be positive integers, let $c,d,\sigma$ be non-negative integers with $c\leq b$ and $d\leq a$,  and let $A$ and $B$ be disjoint sets with $|A|=a$ and $|B|=b$.  If $-\sigma b\leq ac-bd\leq \sigma a$, then there exists a bipartite graph with parts $A$ and $B$ such that every vertex in $A$ has degree between $c-\sigma$ and $c$ and every vertex in $B$ has degree between $d-\sigma$ and $d$.   
\end{lemma}

\begin{proof}
If $0\leq ac-bd\leq \sigma a$, let $q$ and $r$ be the unique integers such that $ac-bd=qa+r$ where $0\leq r\leq a-1$.  Note that $q\leq \sigma$ and $(a-r)(c-q)+r(c-q-1)=bd$ (also note that if $q=\sigma$, then $r=0$).  Furthermore, for all $1\leq k\leq d$, we have $k(c-q)\leq kb$, for all $d\leq k\leq a-r$ we have $k(c-q)\leq (a-r)(c-q)\leq bd$, and for all $a-r\leq k\leq a$ we have $(a-r)(c-q)+(k-(a-r))(c-q-1)\leq (a-r)(c-q)+r(c-q-1)=bd$.  So by Theorem \ref{thm:GR} there is a bipartite graph with parts $A$ and $B$ where $a-r$ vertices in $A$ have degree $c-q$ and $r$ vertices in $A$ have degree $c-q-1$, and every vertex in $B$ has degree $d$.  

If $-\sigma b\leq ac-bd< 0$, let $q$ and $r$ be the unique integers such that $bd-ac=qb+r$ where $0\leq r\leq b-1$.  Note that $q\leq \sigma$ and $(b-r)(d-q)+r(d-q-1)=ac$ (also note that if $q=\sigma$, then $r=0$). In this case we have $(b-r)(d-q)+r(d-q-1)=ac$.  
Furthermore, for all $1\leq k\leq c$, we have $k(d-q)\leq ka$, for all $c\leq k\leq b-r$ we have $k(d-q)\leq (b-r)(d-q)\leq ac$, and for all $b-r\leq k\leq b$ we have $(b-r)(d-q)+(k-(b-r))(d-q-1)\leq (b-r)(d-q)+r(d-q-1)=ac$.  So by Theorem \ref{thm:GR} there is a bipartite graph with parts $A$ and $B$ where $r$ vertices in $B$ have degree $d-q$ and $r$ vertices have degree $d-q-1$, and every vertex in $A$ has degree $c$. 
\end{proof}

For the reader who is specifically interested in the case when $m=2n$ (which is the case relevant to the lower bound on $R(F_n)$) the appendix contains a proof for that case (a number of the calculations become simpler when $m$ is a fixed multiple of $n$).

\begin{example}[General example]
For all integers $m$ and $n$ with $m> n\geq 2$, we have
$$
R(K_{1,m}, F_n)\geq 2\floor{\frac{m+\sqrt{m^2+8n^2}}{2}-n}+2\floor{n-\frac{\sqrt{m^2+8n^2}-m}{4}}-4> \frac{3m+\sqrt{m^2+8n^2}}{2}-8.
$$
In particular, $R(K_{1,2n}, F_n)>\left(3+\sqrt{3}\right)n-8$.
\end{example}

Note that if $m> \frac{2}{9}n^2-9$, then $2m+1> \frac{3m+\sqrt{m^2+8n^2}}{2}-8$ and thus the lower bound of \cite{ZBC3} is better.

\begin{proof}
Let $m$ and $n$ be integers with $m>n\geq 2$.  Define constants $a:=\floor{\frac{m+\sqrt{m^2+8n^2}}{2}-n}-1$ and $b:=\floor{n-\frac{\sqrt{m^2+8n^2}-m}{4}}-1$.  Take disjoint sets $X_1, X_2, Y_1, Y_2$ such that $|X_1|=|X_2|=a$ and $|Y_1|=|Y_2|=b$ and set $N=2a+2b$.  Set $\sigma:=m+n-1-a-2b$ and note that $$\sigma=m+n-\floor{\frac{m+\sqrt{m^2+8n^2}}{2}-n}-2\floor{n-\frac{\sqrt{m^2+8n^2}-m}{4}}+2\geq 2$$
where $\sigma=2$ if and only if $\frac{\sqrt{m^2+8n^2}-m}{4}$ is an integer (which implies $\frac{m+\sqrt{m^2+8n^2}}{2}$ is an integer as well).  Also note that $\sigma<5$ and thus $\sigma\leq 4$.  

Now add all red edges between $X_1$ and $X_2$, between $X_1$ and $Y_2$, and between $X_2$ and $Y_1$.   Between $X_i$ and $Y_i$ we put red edges so that the induced bipartite graph between $X_i$ and $Y_i$ has the property that every vertex in $X_i$ has degree between $n-1-b-\sigma$ and $n-1-b$ and every vertex in $Y_i$ has degree between $n-1-\sigma$ and $n-1$.

\begin{figure}[ht]
    \centering
    \includegraphics[scale=1]{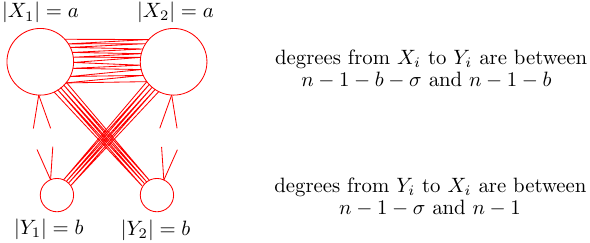}
    \caption{A graph on $N$ vertices with minimum degree at least $N-m$ having no copy of $F_n$.  The values of $a,b,\sigma$ are as in the proof.}
    \label{fig:example}
\end{figure}

Since the minimum degree of the red graph is at least $$a+n-1-\sigma=a+n-1-(m+n-1-a-2b)=2a+2b-m= N-m$$ there is no blue copy of $K_{1,m}$.  Also note that the red graph is tripartite with parts $X_1, X_2, Y_1\cup Y_2$, and every vertex has degree at most $n-1$ to one of the other parts, so there is no red copy of $F_n$. 

All that remains is to show that such a bipartite graph exists between $X_i$ and $Y_i$.  By Lemma \ref{lem:GR} (with $c=n-1-b$ and $d=n-1$), this amounts to checking that $-\sigma b\leq a(n-1-b)-b(n-1)\leq \sigma a.$  Indeed, first suppose $\frac{\sqrt{m^2+8n^2}-m}{4}$ is an integer, which implies $\frac{m+\sqrt{m^2+8n^2}}{2}$ is an integer.  In this case we have $\sigma=2$ and thus
\begin{align*}
&2 a-(a(n-1-b)-b(n-1))\\
&= b(n-1)-a(n-3-b)\\
&= (n-\frac{\sqrt{m^2+8n^2}-m}{4}-1)(n-1)-(\frac{m+\sqrt{m^2+8n^2}}{2}-n-1)(\frac{\sqrt{m^2+8n^2}-m}{4}-2)\\
&=\frac{3\sqrt{m^2+8n^2}+m-8n-2}{2}>\frac{9n+n-8n-2}{2}=n-1>0,
\end{align*}
where the last two inequalities hold since $m> n\geq 2$.  Also we have 
\begin{align*}
&2 b-(b(n-1)-a(n-1-b))\\
&= a(n-1-b)-b(n-3)\\
&= (\frac{m+\sqrt{m^2+8n^2}}{2}-n-1)\frac{\sqrt{m^2+8n^2}-m}{4}-(n-\frac{\sqrt{m^2+8n^2}-m}{4}-1)(n-3)\\
&=4n-3-(\sqrt{m^2+8n^2}-m)\\
&=4n-3-\frac{8n^2}{\sqrt{m^2+8n^2}+m}>4n-3-\frac{8n^2}{4n}=2n-3 > 0 ,
\end{align*}
where the last two inequalities hold since $m> n\geq 2$. 

Now suppose $\frac{\sqrt{m^2+8n^2}-m}{4}$ is not an integer.  In this case we have $\sigma\geq 3$ and thus
\begin{align*}
&3 a-(a(n-1-b)-b(n-1))\\
&= b(n-1)-a(n-4-b)\\
&> (n-\frac{\sqrt{m^2+8n^2}-m}{4}-2)(n-1)-(\frac{m+\sqrt{m^2+8n^2}}{2}-n-1)(\frac{\sqrt{m^2+8n^2}-m}{4}-2)\\
&=\frac{m+ 3 \sqrt{m^2 + 8 n^2}}{2} - 5 n> \frac{n+ 9n}{2} - 5 n = 0,
\end{align*}
where the last inequality holds since $m> n\geq 2$.  Also we have 
\begin{align*}
&3 b-(b(n-1)-a(n-1-b))\\
&= a(n-1-b)-b(n-4)\\
&> (\frac{m+\sqrt{m^2+8n^2}}{2}-n-2)\frac{\sqrt{m^2+8n^2}-m}{4}-(n-\frac{\sqrt{m^2+8n^2}-m}{4}-1)(n-4)\\
&= 5 n-4-\frac{3(\sqrt{m^2 + 8 n^2}-m)}{2}\\
&=5n-4-\frac{12n^2}{\sqrt{m^2 + 8 n^2}+m}>5n-4-\frac{12n^2}{4n}=2n-4\geq 0,
\end{align*}
where the last two inequalities hold since $m> n\geq 2$. 
\end{proof}

\begin{remark}
In terms of the minimum degree version, the lower bound example is roughly as follows.  While we don't repeat the detailed calculations above, we still think it may be useful to formulate the example in this language for future reference.  Let $n$ and $k$ be positive integers with $k<\frac{n}{3}$ and define $\alpha$ so that $k=\alpha n$, and set $\tau=\frac{1+\sqrt{16\alpha^2+1}}{4}-\alpha$.  Let $X_1, X_2, Y_1, Y_2$ be disjoint sets such that $\tau n-1\leq |X_1|,|X_2|\leq \tau n+1$ and $(\frac{1}{2}-\tau)n-1\leq |Y_1|,|Y_2|\leq  (\frac{1}{2}-\tau)n+1$ and $|X_1|+|X_2|+|Y_1|+|Y_2|=n$.  Add all edges between $X_1$ and $X_2$, between $X_1$ and $Y_2$, and between $X_2$ and $Y_1$.  Additionally, for $i\in [2]$, add edges between $X_i$ and $Y_i$ so that every vertex in $Y_i$ has degree between $k-1$ and $k-C$ to $X_i$ and every vertex in $X_i$ has degree between $k-b-1$ and $k-b-C$ to $Y_i$ (for some constant $C$).  Let $G$ be the resulting graph.  We have that $G$ is a tripartite graph with minimum degree $\left(\frac{1+\sqrt{1+16\alpha^2}}{4}\right)n-\Theta(1)$ such that every vertex has degree at most $k-1$ to one of the other parts, so there is no $F_{k}$.
\end{remark}

\subsection{Upper bound}\label{subsec:starfan_upper}

We begin with the Edmonds-Gallai theorem and some related observations.  Given a graph $G$ and a maximum matching $M$, we define the \emph{deficiency} of $G$ as $\mathrm{def}(G)=|V(G)|-2|M|$ (so the deficiency of $G$ is the number of vertices unsaturated by a maximum matching).

\begin{theorem}[Edmonds \cite{Ed}, Gallai \cite{Gal}]\label{thm:GE}
For every graph $G$ there exists $A\subseteq V(G)$ such that if $C$ is the union of the vertex sets of the even components of $G-A$ and $D_1, \dots, D_p$ are the vertex sets of the odd components of $G-A$, then for every maximum matching $M$ in $G$,
\begin{enumerate}
\item $M\cap E(G[C])$ is a perfect matching of $G[C]$,
\item $M$ matches $A$ to distinct odd components of $G-A$, and
\item for all $i\in [p]$, $M\cap E(G[D_i])$ is a near-perfect matching of $G[D_i]$
\end{enumerate}
and consequently
\begin{enumerate}[resume]
\item $p=|A|+\mathrm{def}(G)$, and
\item $\nu(G)=|A|+\frac{1}{2}(|C|+\sum_{i\in [p]}(|D_i|-1)).$
\end{enumerate}
\end{theorem}

Given a graph $G$, we call the collection of sets $\{A, C, D_1, \dots, D_p\}$ guaranteed by Theorem \ref{thm:GE} an \emph{Edmonds-Gallai partition}\footnote{Note that strictly speaking $\{A, C, D_1, \dots, D_p\}$ may not be a partition since we could have $A=\emptyset$ or $C=\emptyset$.  Also note that if $p\geq 1$, then $D_i\neq \emptyset$ for all $i\in [p]$.} of $G$.

Tailoring the Edmonds-Gallai theorem to our specific setting (which, for reasons to be explained later, will be a 2-colored complete graph where every vertex has monochromatic degree between $2n$ and $3n$), we obtain the following structural result by applying Theorem \ref{thm:GE} inside the monochromatic neighborhood of any vertex.  

\begin{lemma}[Edmonds-Gallai structure]\label{lem:structure}
Let $G$ be a 2-colored complete graph on $N$ vertices with $2n\leq N< 3n$ and let $\chi\in [2]$.  If $\nu(G_\chi)\leq n-1$ and $\{A, C, D_1, \dots, D_p\}$ is an Edmonds-Gallai partition of $G_\chi$, then 
\begin{enumerate}
\item $|A|+\frac{1}{2}(|C|+\sum_{i\in [p]}(|D_i|-1))=\nu(G_\chi)\leq n-1$,
\item $p=|A|+\mathrm{def}(G_\chi)\geq |A|+N-(2n-2)$, and
\item for all $1\leq i<j\leq p$, every edge between $D_i$ and $D_j$ is color $3-\chi$, and every edge from $D_i$ to $C$ is color $3-\chi$.
\end{enumerate}
\end{lemma}

\begin{proof}
The proof follows immediately from Theorem \ref{thm:GE} noting that $\mathrm{def}(G_\chi)=N-2\nu(G_\chi)\geq N-(2n-2)$.
\end{proof}



We now prove the upper bound in Theorem \ref{thm:fanstar}, which says that for all $m> n\geq 2$, $R(K_{1,m}, F_{n})< \frac{3m+\sqrt{m^2+8n^2}}{2}+1$.

\begin{proof}[Proof of Theorem \ref{thm:fanstar} (upper bound)]
Let $m$ and $n$ be integers with $m>n\geq 2$.  Let\footnote{This parametrization allows us to write $m$, and consequently $N$, as a multiple of $n$, but note that $\gamma$ may depend on $n$ if $m=\omega(n)$.} $\gamma=\frac{m}{n}$ and let $N= \ceiling{(\frac{3\gamma+\sqrt{\gamma^2+8}}{2})n}$. Suppose we have a 2-coloring of $K_N$ with no blue copy of $K_{1,\gamma n}$; that is, every vertex has blue degree at most $\gamma n-1$ and consequently, every vertex has red degree at least $N-1-(\gamma n-1)=N-\gamma n$.  We may further suppose that the 2-coloring is maximal with respect to not having a blue copy of $K_{1,\gamma n}$ which implies that 
\begin{equation}\label{eq:blue2}
\text{for all } uv\in E(G_R), d_R(u)=N-\gamma n \text{ or } d_R(v)=N-\gamma n.
\end{equation}
Indeed, if $uv$ is a red edge with $d_R(u)\geq N-\gamma n+1$ and $d_R(v)\geq N-\gamma n+1$, then $u$ and $v$ both have blue degree at most $N-1-(N-\gamma n+1)=\gamma n-2$ and thus we could make $uv$ blue without creating a blue copy of $K_{1,\gamma n}$.  

Now let $G$ be the graph consisting only of the red edges, so $G$ has $N$ vertices and by \eqref{eq:blue2} we have
\begin{equation}\label{eq:N-g}
\delta(G)= N-\gamma n=\ceiling{(\frac{\gamma+\sqrt{\gamma^2+8}}{2})n}. 
\end{equation}
For the rest of this proof, we set $\delta:=\frac{\gamma+\sqrt{\gamma^2+8}}{2}$.  Note that 
\begin{equation}\label{eq:root}
\delta^2-\delta\gamma-2=0.
\end{equation}

Suppose for contradiction that there is no red copy of $F_n$; that is, $G$ contains no copy of $F_n$.  We claim that 
\begin{equation}\label{eq:1+g}
\Delta(G)<(1+\gamma)n.  
\end{equation}
Indeed, if there exists $v\in V(G)$ such that $d(v)\geq (1+\gamma)n>2n$, then $\delta(G[N(v)])\geq d(v)-\gamma n\geq n$ and thus $G[N(v)]$ contains a matching of size at least $n$ which gives a copy of $F_n$.  Define $\Delta$ so that $\Delta(G)=\Delta n$.

We now use Theorem \ref{thm:GE} to establish some structural information about the graph induced by the neighborhood of an arbitrary vertex.  Let $v\in V(G)$ and let $G_v=G[N(v)]$ and note that 
\begin{equation}\label{eq:delta1}
\delta(G_v)\geq |V(G_v)| -\gamma n\geq \delta(G)-\gamma n.  
\end{equation}
We may assume that a largest matching $M_v$ in $G_v$ has size less than $n$ (otherwise $G$ contains a copy of $F_n$) and thus by applying Theorem \ref{thm:GE} to $G_v$, we get an Edmonds-Gallai partition $\{A, C, D_1, \dots, D_{p_v}\}$ with $|D_1|\geq |D_2|\geq \dots \geq |D_{p_v}|$ and
\begin{equation}\label{eq:p}
p_v=\mathrm{def}(G_v)+|A|=|V(G_v)|-2|M_v|+|A|\geq |V(G_v)|-2n+|A|+2
\end{equation}
and
\begin{equation}\label{eq:excess}
|C|+\sum_{i\in [p_v]}(|D_i|-1)=2(|M_v|-|A|)\leq 2(n-1-|A|).
\end{equation}

Now we establish the following claim.

\begin{claim}\label{clm:q=1}
$|D_{p_v}|=1$.
\end{claim}

\begin{proof}
Set $q:=|D_{p_v}|$ and suppose for contradiction that $q\geq 3$.  Note that since $|V(G_v)|=|A|+|C|+\sum_{i\in [p_v]}|D_i|\geq p_vq$, we have $q\leq \frac{|V(G_v)|}{p_v}\stackrel{\eqref{eq:p}}{\leq} \frac{|V(G_v)|}{\mathrm{def}(G_v)}$.  We first show that 
\begin{equation}\label{eq:deltaq}
\delta(G_v)\geq q.
\end{equation}
Suppose to the contrary that $\delta(G_v)\leq q-1\leq \frac{|V(G_v)|}{\mathrm{def}(G_v)}-1$, which rearranges to
\begin{align*}
|V(G_v)|\geq \mathrm{def}(G_v)(\delta(G_v)+1)&\stackrel{\mathclap{\eqref{eq:delta1},\eqref{eq:p}}}{\geq} (|V(G_v)|-2n+2)(|V(G_v)|-\gamma n+1)\\
&= |V(G_v)|+|V(G_v)|(|V(G_v)|-\gamma n)-(2n-2)(|V(G_v)|-\gamma n+1)\\
&\geq  |V(G_v)|+\delta(G)(\delta(G)-\gamma n)-(2n-2)(\Delta(G)-\gamma n+1)\\
&\stackrel{\mathclap{\eqref{eq:root},\eqref{eq:1+g}}}{>}|V(G_v)|+2n^2-(2n-2)(n+1)=|V(G_v)|+2
\end{align*}
a contradiction.   

Note that for all $v\in D_{p_v}$, we have $N_{G_v}(v)\subseteq D_{p_v}\cup A$ and thus 
\begin{align*}
|V(G_v)|\geq p_vq+|A|&\geq q(\mathrm{def}(G_v)+|A|)+|A|\\
&=q\cdot\mathrm{def}(G_v)+(q+1)|A|\\
&\geq q\cdot\mathrm{def}(G_v)+(q+1)(\delta(G_v)-(q-1))=:f(q),
\end{align*}
which is a contradiction provided $f(q)>|V(G_v)|$.  To see this, note that since $f(q)$ is a concave quadratic in $q$, we only need to check $f(q)>|V(G_v)|$ at the endpoints.  If $q = \frac{|V(G_v)|}{\mathrm{def}(G_v)}$, then $f(\frac{|V(G_v)|}{\mathrm{def}(G_v)})=|V(G_v)| + (q+1)(\delta(G_v) - q + 1)\stackrel{\eqref{eq:deltaq}}{>}|V(G_v)|$.  If $q=3$, then 
\begin{align*}
f(3)-|V(G_v)|&=3\mathrm{def}(G_v)+4(\delta(G_v)-2)-|V(G_v)|\\
&\stackrel{\mathclap{\eqref{eq:delta1},\eqref{eq:p}}}{\geq} 3(|V(G_v)|-2n+2)+4(|V(G_v)|-\gamma n-2)-|V(G_v)|\\
&=6|V(G_v)|-6n-4\gamma n-2\\
&\geq 6\delta(G)-6n-4\gamma n-2\\
&\geq 3(\gamma+\sqrt{\gamma^2+8})n-6n-4\gamma n-2\\ 
&=(3\sqrt{\gamma^2+8}-\gamma-6)n-2\\
&> 2n-2 \geq 0 
\end{align*}
where the last two inequalities holds since $\gamma>1$ and $n\geq 2$ (and $3\sqrt{\gamma^2+8}-\gamma$ is increasing).

Thus we have $q<3$ and since $q$ is odd, this implies that $q=1$.  
\end{proof}

Having established these general properties we now proceed with the main part of the argument. Let $v_1\in V(G)$ with $d(v_1)=\Delta(G)$ and let $G_1=G[N(v_1)]$, so 
\begin{equation}\label{eq:V1}
|V(G_1)|= \Delta(G).
\end{equation}
Let $\{A, C, D_1, \dots, D_{p_1}\}$ be the Edmonds-Gallai partition of $G_1$ and note that by Claim \ref{clm:q=1} and \eqref{eq:delta1} we have
\begin{equation}\label{eq:|A|}
n-1\geq |A|\geq \delta(G_1)\geq (\Delta-\gamma)n.
\end{equation}
Set $X_1=C\cup D_1\cup  \dots \cup D_{p_1}$ and note that
\begin{equation}\label{eq:x1-p1}
|X_1|-p_1=|C|+\sum_{i\in [p_1]}(|D_i|-1)\stackrel{\eqref{eq:excess}}{\leq} 2(n-1-|A|).
\end{equation}

Now let $v_2\in D_{p_1}$ (recall that $|D_{p_1}|=1$ by Claim \ref{clm:q=1}) and let $G_2=G[N(v_2)]$.  Note that $N(v_2)\cap X_1=\emptyset$ and by \eqref{eq:blue2} we have 
\begin{equation}\label{eq:V2}
|V(G_2)|= \delta(G).
\end{equation}

Let $\{A', C', D_1', \dots, D_{p_2}'\}$ be the Edmonds-Gallai partition of $G_2$ and again by Claim \ref{clm:q=1} and \eqref{eq:delta1} we have 
\begin{equation}\label{eq:|A'|}
n-1\geq |A'|\geq \delta(G_2)\geq \ceiling{(\delta-\gamma)n}.
\end{equation}
Set $X_2=C'\cup D_1'\cup  \dots \cup D_{p_2}'$ and note that 
\begin{equation}\label{eq:x2-p2}
|X_2|-p_2=|C'|+\sum_{i\in [p_2]}(|D_i'|-1)\stackrel{\eqref{eq:excess}}{\leq} 2(n-1-|A'|).
\end{equation}

Now to complete the proof, we count the edges between $X_1\cup X_2$ and $Y:=V(G)\setminus (X_1\cup X_2)$ which will give us a vertex in $Y$ whose degree is greater than $\Delta(G)$ (that is, we will show $e(X_1\cup X_2, Y)>\Delta(G)|Y|$), a contradiction.

\begin{figure}[ht]
    \centering
    \includegraphics[scale=1]{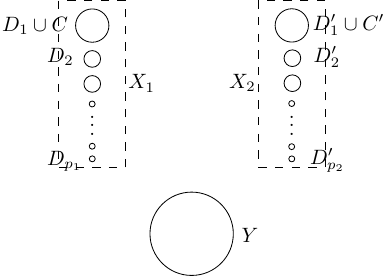}
    \caption{Counting the edges between $X_1\cup X_2$ and $Y=V(G)\setminus (X_1\cup X_2)$}
    \label{fig:starfan}
\end{figure}

\begin{claim}\label{clm:worstcase}
$
e(X_1 \cup X_2, Y) \geq p_1(\delta(G) - |X_2|) - (|X_1| - p_1) + p_2(\delta(G) - |X_1|) - (|X_2| - p_2)
$
\end{claim}

\begin{proofclaim}
We establish the claim by showing that for all $i\in [2]$, $e(X_i, Y) \geq p_i(\delta(G) - |X_{3-i}|) - (|X_i| - p_i).$  Without loss of generality, suppose $i=1$.  

We note that if there are any even components (that is, $C\neq \emptyset$), then for the purposes of this proof, we can just add these vertices to $D_1$; formally, reset $D_1:=D_1\cup C$.  For all $i\in [p_1]$, since $N_{G}(u)\subseteq D_i\cup X_2\cup Y$ for all $u\in D_i$, we have $$e(D_i, Y)\geq  |D_i| \max\{0, \delta(G) - |X_2| - |D_i| + 1\}.$$  Let $c$ be maximum such that $|D_c|\geq \delta(G)-|X_2|+1$.  For all $i\in [c]$, we have the trivial lower bound $e(D_i, Y) \geq 0$.  For all $i\in [p_1]\setminus [c]$ we have $1 \leq |D_i| \leq \delta(G) - |X_2|$ and thus 
$$|D_i|(\delta(G) - |X_2| - |D_i| + 1) - (\delta(G) - |X_2|) = (|D_i| - 1)(\delta(G) - |X_2| - |D_i|) \geq 0,$$ which implies  
\begin{equation}\label{eq:smallD}
e(D_i, Y) \geq \delta(G) - |X_2|. 
\end{equation}

Note that by the definition of $c$ we have 
\begin{equation}\label{eq:c()}
c(\delta(G) - |X_2|) \leq \sum_{i=1}^{p_1} (|D_i| - 1) = |X_1| - p_1. 
\end{equation}
Thus, by \eqref{eq:smallD} the total number of edges from $X_1$ to $Y$ is 
\begin{align*}
e(X_1, Y) \geq (p_1 - c)(\delta(G) - |X_2|) &= p_1(\delta(G) - |X_2|) - c(\delta(G) - |X_2|)\\
&\stackrel{\eqref{eq:c()}}{\geq} p_1(\delta(G) - |X_2|) - (|X_1| - p_1). \qedhere
\end{align*}
\end{proofclaim}

Define $\alpha$ and $\alpha'$ so that $|A|=(1-\alpha)n$ and $|A'|=(1-\alpha')n$.  
Note that by \eqref{eq:|A|} and \eqref{eq:|A'|}, we have that 
\begin{equation}\label{eq:1/n}
\alpha, \alpha'\geq \frac{1}{n}.
\end{equation}
By Claim \ref{clm:worstcase} together with \eqref{eq:p}, \eqref{eq:V1}, \eqref{eq:V2}, \eqref{eq:x1-p1}, \eqref{eq:x2-p2}  we have
\begin{align*}
e(X_1\cup X_2,Y) &\geq p_1(\delta(G)-|X_2|)-(|X_1|-p_1)+p_2(\delta(G)-|X_1|)-(|X_2|-p_2)\\
&\geq ((\Delta-2)n+|A|)(\delta(G)-(\delta(G)-|A'|))-2(\alpha n-1)\\
&\qquad\qquad\qquad\qquad\qquad\qquad +((\delta-2)n+|A'|)(\delta n-(\Delta n-|A|))-2(\alpha' n-1)\\
&=(\Delta-1-\alpha)(1-\alpha')n^2+(\delta-1-\alpha')(1-\alpha-(\Delta-\delta))n^2-2(\alpha n-1)-2(\alpha' n-1).
\end{align*}
We also have
\begin{align*}
|Y|=N-|X_1|-|X_2|&= N-(\Delta(G)-|A|)-(\delta(G)-|A'|)\\
&\stackrel{\eqref{eq:N-g}}{=} N-(\Delta(G)-|A|)-(N-\gamma n-|A'|)\\
&=(2+\gamma-\Delta-\alpha-\alpha')n.
\end{align*}
Putting this all together, we have
\begin{align*}
&e(X_1\cup X_2, Y)-\Delta(G)|Y|\\
&\geq (\Delta-1-\alpha)(1-\alpha')n^2+(\delta-1-\alpha')(1-\alpha-(\Delta-\delta))n^2-2(\alpha n-1)-2(\alpha' n-1) \\
&\quad  -\Delta n(2+\gamma-\Delta-\alpha-\alpha')n\\
&\stackrel{\eqref{eq:root}}{=}(\Delta-\delta)(\Delta-\gamma + \alpha + \alpha')n^2+2(\alpha n - 1)(\alpha'n - 1) + 2\stackrel{\eqref{eq:1/n}}{>} 0,
\end{align*}
and thus $\frac{e(X_1\cup X_2, Y)}{|Y|}>\Delta(G)$, a contradiction.
\end{proof}

\section{Upper bounds on the Ramsey numbers of fans}\label{sec:fan_upper}

In this section we prove the upper bound in Theorem \ref{thm:main}; that is, $R(F_n)\leq (5+o(1))n$.  We begin with a number of lemmas.

\subsection{Lemmas for fans}

In the course of proving $R(F_n)\leq 6n$, Lin and Li \cite{LL} proved that $R(nK_2, F_{n})=3n$.  Despite being used to obtain a weaker upper bound on $R(F_n)$, this result is still useful for us because it implies that in every 2-colored complete graph, if there is a vertex incident with at least $3n$ edges of the same color, then we have a monochromatic copy of $F_n$.

\begin{corollary}\label{cor:LL}
Let $K$ be a 2-colored complete graph.  If there exists $v\in V(K)$ and $i\in [2]$ such that $|N_i(v)|\geq 3n$, then $K$ contains a monochromatic $F_n$.
\end{corollary}

The following lemma describes the size of a largest matching in a complete multipartite graph broken into cases based on the size and number of the parts.  Note that given a vertex $v$ in a complete multipartite graph $G$, we have that $G[N(v)]$ is also a complete multipartite graph and thus by extension, the following lemma also tells us the size of a largest fan in a complete multipartite graph.  In fact, we state something more general regarding cycles as this is potentially more applicable for future use.

\begin{theorem}[Bondy \cite{Bon}]\label{thm:bon}
If $G$ is a graph on $n\geq 3$ vertices with $\delta(G)\geq \frac{n}{2}$, then either $G$ is pancyclic or $G$ is isomorphic to $K_{n/2, n/2}$.
\end{theorem}

\begin{lemma}[Matchings/cycles in complete multipartite graphs]\label{lem:partite_matching}
Let $G$ be a complete multipartite graph on $N$ vertices with parts $V_1, \dots, V_t$ such that $|V_1|\leq \dots \leq |V_t|$.
\begin{enumerate}
\item If $t\geq 3$ and $|V_t|\leq \frac{N}{2}$, then $G$ is pancyclic.  In particular, $G$ has a matching covering $2\floor{\frac{N}{2}}\geq N-1$ vertices of $G$.
\item If $t\geq 3$ and $|V_t|>\frac{N}{2}$, then $G$ has a cycle of length $\ell$ for all $3\leq \ell\leq 2(N-|V_t|)$.  In particular, $G$ has a matching covering $2(N-|V_t|)$ vertices of $G$.
\item If $t=2$, then $G$ has a cycle of length $2\ell$ for all $4\leq 2\ell\leq 2|V_1|$.  In particular, $G$ has a matching covering $2|V_1|$ vertices of $G$.
\end{enumerate}
\end{lemma}

\begin{proof}
(i) Since $t\geq 3$, we have that $G$ is not bipartite and since $|V_t|\leq \frac{N}{2}$, we have $\delta(G)\geq N-|V_t|\geq \frac{N}{2}$.  So by Theorem \ref{thm:bon}, $G$ is pancyclic.

(ii) Let $V_t'\subseteq V_t$ with $|V_t'|=\sum_{i=1}^{t-1}|V_i|=N-|V_t|$ and then apply (i) to $G'=G[V_1\cup \dots \cup V_{t-1}\cup V_t']$ (note that $|V(G')|=|V_1|+\dots +|V_{t-1}|+|V_t'|=2(N-|V_t|)$).

(iii) This follows directly since $G$ is a complete bipartite graph.
\end{proof}

Throughout the proof of Theorem \ref{thm:main} we will be in a situation where we have a monochromatic complete multipartite graph which already contains a large fan, but not quite large enough.  To obtain a fan of the desired size, it suffices to find a vertex in the multipartite graph whose neighborhood contains a sufficiently large matching extending into the rest of the graph.  The following technical lemma describes various numerical conditions which need to be satisfied in order to get a fan of the desired size.

\begin{lemma}[Fan extension lemma]\label{lem:main}
Let $n$ be a positive integer and let $\lambda$ be a real number with $\lambda\geq 1$. Let $G$ be a graph and let $\{X_1, \dots, X_p,Y,Z\}$ be a partition of $V(G)$ such that $G[X_1\cup \dots \cup X_p\cup Y]$ is a complete multipartite graph. Set $X:=X_1\cup \dots \cup X_p$ and $q:=2n-(|X|+|Y|)$.  Finally, suppose that $|X|+|Y|>n$ and $|X_i|\leq \lambda$ for all $i\in [p]$. 
\begin{enumerate}
\item If $|X|> n+\lambda$ and there exists $v\in X$ such that there is a matching in $G[N(v)\cap Z, X\cup Y]\cup G[N(v)\cap Z]$ covering more than $q+2\lambda$ vertices of $Z$, then $G$ contains a copy of $F_n$.

\item If $|Y|\leq n$ and there exists $v\in X$ such that there is a matching in $G[N(v)\cap Z, Y]\cup G[N(v)\cap Z]$ covering more than $2(q+\lambda)$ vertices of $Y\cup Z$, then $G$ contains a copy of $F_n$.

\item If $|Y|\geq n$ and there exists $v\in X$ such that there is a matching  in $G[N(v)\cap Z, Y]\cup G[N(v)\cap Z]$ covering at least $2(n-|X|+\lambda)$ vertices of $Y\cup Z$, then $G$ contains a copy of $F_n$.
\end{enumerate}
\end{lemma}

\begin{figure}[ht]
     \centering
     \begin{subfigure}[b]{0.3\textwidth}
         \centering
         \includegraphics[width=\textwidth]{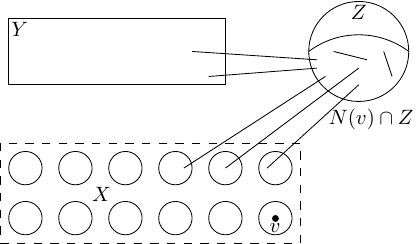}
         \caption{Case (i)}
         \label{fig:1a}
     \end{subfigure}
     \hfill
     \begin{subfigure}[b]{0.3\textwidth}
         \centering
         \includegraphics[width=\textwidth]{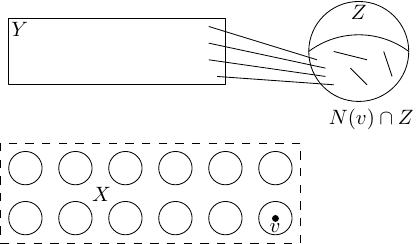}
         \caption{Case (ii)}
         \label{fig:1b}
     \end{subfigure}
     \hfill
     \begin{subfigure}[b]{0.3\textwidth}
         \centering
         \includegraphics[width=\textwidth]{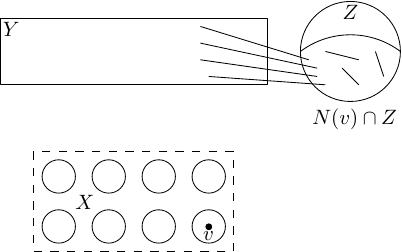}
         \caption{Case (iii)}
         \label{fig:1c}
     \end{subfigure}
        \caption{Fan extension lemma}
        \label{fig:1}
\end{figure}

\begin{proof} 
(i) Let $i\in [p]$ and $v\in X_i$ such that there is a matching in $G_v:=G[N(v)\cap Z, X\cup Y]\cup G[N(v)\cap Z]$ covering more than $q+2\lambda$ vertices of $Z$. First suppose $q \geq 0$ (meaning $|X|+|Y|\leq 2n$). Let $M$ be a matching in $G_v-X_i$ covering at least $q+|X_i|+1$ vertices of $Z$ and at most $q+|X_i|+1$ vertices of $X\cup Y$ (this is possible since $\lceil 2\lambda\rceil\geq 2\lambda\geq 2\lfloor\lambda\rfloor\geq 2|X_i|$), and let $U=V(M)\cap (X\cup Y)$. Note that by the choice of $M$, $U\cap X_i=\emptyset$. By the case assumption and the fact that $|X_i|\leq \lambda$, we have $|X|-|X_i|-|U|\geq |X|-|X_i|-(q+|X_i|+1)=2|X|-2n-2|X_i|-1+|Y|\geq |Y|$. So $G[Y\cup (X\setminus X_i)\setminus U]$ induces a complete multipartite graph where $|(X\setminus X_i)\setminus U|\geq |Y|$. Thus by Lemma \ref{lem:partite_matching}(i) we have a matching in $G[N(v)]$ covering at least $|X|+|Y|-|X_i|-1+q+|X_i|+1= 2n$ vertices.

Suppose instead that $q < 0$ (meaning $|X|+|Y|>2n$). Since $|X|> n+\lambda$, we have the desired fan entirely within $X\cup Y$ regardless of the size of $Y$ by Lemma \ref{lem:partite_matching}(i) or (ii).

(ii) Let $i\in [p]$ and $v\in X_i$ such that there is a matching in $G_v:=G[N(v)\cap Z, Y]\cup G[N(v)\cap Z]$ covering more than $2(q+\lambda)$ vertices of $Y\cup Z$. Let $k := \max\{0, q+|X_i|+1\}$. Let $M$ be a matching in $G_v$ of size exactly $k$, and let $U=V(M)\cap Y$. If $|Y|\geq k$, then define a set $U'$ such that $U\subseteq U'\subseteq Y$ and $|U'|=k$. Note that in this case, since $k \geq q+|X_i|+1$ and $|Y|\leq n$, we have:$$|Y|-|U'| = |Y|-k \leq |Y|-(q+|X_i|+1) = 2|Y|-2n+|X|-|X_i|-1 \leq |X|-|X_i|$$If $|Y|<k$, then set $U'=Y$. Either way, we have that $G[(Y\setminus U')\cup (X\setminus X_i)]$ induces a complete multipartite graph where $|Y|-|U'|\leq |X|-|X_i|$. Thus by Lemma \ref{lem:partite_matching}(i), we have a matching in $G[N(v)]$ covering at least $$k+|X|+|Y|-|X_i|-1\geq q+|X_i|+1+|X|+|Y|-|X_i|-1=2n$$ vertices.

(iii) Let $i\in [p]$ and $v\in X_i$ such that there is a matching in $G_v:=G[N(v)\cap Z, Y]\cup G[N(v)\cap Z]$ covering at least $2(n-|X|+\lambda)$ vertices of $Y\cup Z$.  Let $M$ be a matching in $G_v$ of size exactly $n-|X|+|X_i|$ and let $U=V(M)\cap Y$.  Note that since $|Y|\geq n$, we have $$|Y|-|U|\geq |Y|-(n-|X|+|X_i|)=|X|+|Y|-n-|X_i|\geq |X|-|X_i|.$$  So $G[(Y\setminus U)\cup (X\setminus X_i)]$ induces a complete multipartite graph where $|Y|-|U|\geq |X|-|X_i|$.  Thus by Lemma \ref{lem:partite_matching}(ii) we have a matching in $G[N(v)]$ covering at least $2(n-|X|+|X_i|)+2(|X|-|X_i|)=2n$ vertices.
\end{proof}

We will use the following variant of K\"onig's theorem (see \cite{Aha} or \cite{D6}).

\begin{theorem}[K\"onig]\label{thm:konig}
For every bipartite graph $G$ and every maximum matching $M$ in $G$, there exists a minimum vertex cover $Q$ of $G$ such that for all $e\in M$, $|e\cap Q|=1$.  
\end{theorem}

We now state a lemma of Chen, Yu, and Zhao \cite[Lemma 1.3]{CYZ} for the main purpose of explaining the generalization which follows.

\begin{lemma}[Chen, Yu, Zhao \cite{CYZ}]\label{lem:CYZ}
Let $n$ be an integer with $n\geq 2$, let $K$ be a 2-colored complete graph, and let $V_0$ be a red clique in $K$ with $|V_0|\geq \frac{3n}{2}+1$.  If every vertex in $V_0$ has red degree at least $dn\geq \frac{5n}{2}$, then $\bigcup_{v\in V_0}N_R(v)$ contains a monochromatic $F_{n}$ with center in $V_0$.
\end{lemma}

We now generalize Lemma \ref{lem:CYZ} by (i) allowing for different lower bounds on $|V_0|$ and $dn$ (in particular, if $|V_0|$ is larger then we can allow for $dn$ to be smaller, and vice versa), (ii) allowing for $V_0$ to be a complete multipartite graph (with bounded part sizes) instead of a clique, and (iii) allowing for the off-diagonal case where we either get a red $F_n$ or a blue $F_{n-p}$.  We note that Dvo{\v{r}}{\'a}k and Metrebian \cite{DM} also generalize Lemma \ref{lem:CYZ}, but as these are written so differently, it is hard for us to say whether these two generalizations are equivalent.  

\begin{lemma}\label{lem:cover_gen_partite}
Let $\ep>0$, let $d\in \mathbb{R}$ with $d>2$, let $n\geq \frac{2}{(d-2)\ep}$, and let $0\leq p<n$.  Let $K$ be a 2-colored complete graph and let $V_0$ be a complete multipartite subgraph of color $\chi\in [2]$ where the largest part has cardinality $\ell\leq \frac{1}{\ep}$, and $n+2\ell< |V_0|< 2n+\ell$. Let $k$ be the largest integer such that $k(2n+2\ell-|V_0|)<|V_0|$.  If every vertex in $V_0$ has degree at least $dn$ in color $\chi$ and 
\begin{equation}\label{eq:k'_gen}
|V_0|\geq (1+\frac{1}{k})(2n-2p)-(k+1)((d-2)n-2\ell),
\end{equation}
then $\bigcup_{v\in V_0}N_\chi(v)$ either contains a monochromatic $F_{n}$ of color $\chi$ with center in $V_0$ or a monochromatic $F_{n-p}$ of color $3-\chi$ with center in $V_0$.
\end{lemma}

Note that Lemma \ref{lem:CYZ} corresponds to the case when $p=0$, $\ep=\ell=1$, $d=2.5$, and $|V_0|\geq \frac{3n}{2}+1$ (so $k\geq 3$).

\begin{proof}
Without loss of generality suppose $V_0$ is a blue complete multipartite graph; that is, color $\chi$ is blue and color $3-\chi$ is red.  

For all $v\in V_0$, we consider the blue graph $G_B(v):=G_B[N_B(v)\setminus V_0]\cup G_B[N_B(v)\setminus V_0, V_0]$.  Let $M$ be a maximum blue matching in $G_B(v)$ with the secondary property that it has as many edges as possible inside $G_B[N_B(v)\setminus V_0]$.  By Lemma \ref{lem:main}(i) (with $Y=\emptyset$), if $M$ covers at least $2n-|V_0|+2\ell+1$ vertices of $N_B(v)\setminus V_0$, then we have a blue $F_n$; so suppose not.  Because $M$ was chosen to maximize the number of edges inside $G_B[N_B(v)\setminus V_0]$, there can be no internal edges disjoint from $M$ (otherwise we could swap it with a cross-edge to strictly increase the internal edge count). Thus, taking both endpoints of the internal edges of $M$, and applying Theorem \ref{thm:konig} to the remaining cross-edges, we obtain a vertex cover of $G_B(v)$ of cardinality at most $2n-|V_0|+2\ell$.

Assume that for all $v\in V_0$ there is no blue $F_n$ with $v$ as the center.  For all $v\in V_0$, choose a minimum vertex cover $Q_v$ of $G_B(v)$ with the additional property that $x_v:=|Q_v\cap (N_B(v)\setminus V_0)|$ is minimized and set $y_v:=|Q_v|-x_v$.  Note that $(N_B(v)\setminus V_0)\setminus Q_v$ induces a red clique.

Now, over all vertices in $V_0$, choose $v_1$ so that $y_{v_1}$ is maximized.  Then choose $v_2\in V_0\setminus Q_{v_1}$ so that $y_{v_2}$ is maximized and note that by the property of $Q_{v_1}$, we have that $v_2$ has no blue neighbors in $(N_B(v_1)\setminus V_0)\setminus Q_{v_1}$.  Then choose $v_3\in V_0\setminus (Q_{v_1}\cup Q_{v_2})$ so that $y_{v_3}$ is maximized and again note that $v_3$ has no blue neighbors in $((N_B(v_1)\setminus V_0)\setminus Q_{v_1})\cup ((N_B(v_2)\setminus V_0)\setminus Q_{v_2})$.  Continue this process until the point at which we have a vertex $v_{k'+1}\in V_0\setminus (\bigcup_{i=1}^{k'}Q_{v_i})$ such that $V_0\subseteq \bigcup_{i=1}^{k'+1} (Q_{v_i}\cap V_0)$.  Note that the sets $Q_{v_1}\cap V_0, \dots, Q_{v_{k'}}\cap V_0$ are not necessarily disjoint; however, all of them have size at most $2n+2\ell-|V_0|$ so we have $k'\geq k$ by the definition of $k$.  By the choice of $v_1, \dots, v_{k'}$, we have that $y_{v_i}\geq y_{v_{k'+1}}$ for all $i\in [k]$ and thus 
\begin{equation}\label{eq:yi}
\sum_{i=1}^{k'}y_i\geq k'\cdot y_{k'+1}.
\end{equation}

Let $A:=\bigcup_{i=1}^{k'}((N_B(v_i)\setminus V_0)\setminus Q_{v_i})$ and note that $A\subseteq N_R(v_{k'+1})$.  If $|A|\geq 2n-2p+k'$, then since $A$ contains $k'$ disjoint red cliques (the sets $(N_B(v_i)\setminus V_0)\setminus Q_{v_i}$, which are pairwise disjoint because each $v_i$ was chosen to have no blue neighbors in the preceding sets) with a total of at least $2n-2p+k'$ vertices, there is a red matching of size at least $n-p$ in $A$ which, together with $v_{k'+1}$, gives us the desired red copy of $F_{n-p}$.  We now show that it is indeed the case that $|A|\geq 2n-2p+k'$.  

Since every vertex in $V_0$ has at least $dn-(|V_0|-1)$ blue neighbors outside of $V_0$, we have
\begin{align}
|A|\geq k'(dn-(|V_0|-1))-\sum_{i=1}^{k'}x_i&\geq k'(dn-(|V_0|-1))-(k'(2n+2\ell-|V_0|)-\sum_{i=1}^{k'}y_i)\notag\\
&=k'((d-2)n-(2\ell-1))+\sum_{i=1}^{k'}y_i\label{eq:A}
\end{align}

If $y_{k'+1}\geq \frac{2n-2p}{k'}-(d-2)n+2\ell$, then by \eqref{eq:A} and \eqref{eq:yi} we have 
\begin{align*}
|A|\geq k'((d-2)n-(2\ell-1))+\sum_{i=1}^{k'}y_i&\geq k'((d-2)n-(2\ell-1))+k'(\frac{2n-2p}{k'}-(d-2)n+2\ell)\\
&= 2n-2p+k' 
\end{align*}
as desired.  

Otherwise, we have $y_{k'+1} < \frac{2n-2p}{k'}-(d-2)n+2\ell$ and since $V_0\subseteq \bigcup_{i=1}^{k'+1} (Q_{v_i}\cap V_0)$ we have
\begin{equation}\label{eq:yi2}
\sum_{i=1}^{k'}y_i \geq |V_0|-y_{k'+1} > |V_0|-(\frac{2n-2p}{k'}-(d-2)n+2\ell).
\end{equation}
So by \eqref{eq:A} and \eqref{eq:yi2} we have 
\begin{align*}
|A| &> k'((d-2)n-(2\ell-1))+|V_0|-(\frac{2n-2p}{k'}-(d-2)n+2\ell)\\
&= (k'+1)((d-2)n-2\ell) + k' + |V_0| - \frac{2n-2p}{k'}\\
&\stackrel{\eqref{eq:k'_gen}}{\geq} (k'+1)((d-2)n-2\ell) + k' + (1+\frac{1}{k})(2n-2p)-(k+1)((d-2)n-2\ell)-\frac{2n-2p}{k'}\\
&= (2n-2p)(1+\frac{1}{k}-\frac{1}{k'}) + k' + (k'-k)((d-2)n-2\ell)\\
&\geq 2n-2p+k',
\end{align*}
(where we used $k'\geq k$ in the last inequality) which completes the proof.
\end{proof}

The following corollary is tailored to our specific application.

\begin{corollary}\label{cor:main}
Let $\ep>0$, $0\leq \alpha<\frac{1}{2}-\ep$, and let $n\geq \frac{3}{\ep^2}$.  Let $K$ be a 2-colored complete graph, let $V_0$ be a monochromatic complete multipartite subgraph of $K$ of color $\chi\in [2]$ with largest part of cardinality $\ell\leq \frac{1}{\ep}$ and $n+2\ell< |V_0|< 2n+\ell$, let $U_0$ be a monochromatic complete multipartite subgraph of $K$ of color $3-\chi$, disjoint from $V_0$, with parts of size at most $\frac{1}{\ep}$, and suppose that every edge between $U_0$ and $V_0$ is color $3-\chi$.  Let $k$ be the largest integer such that $k(2n+2\ell-|V_0|)<|V_0|$.  If every vertex in $V_0$ has degree at least $(2.5-\alpha+\ep)n$ in color $\chi$ and 
\begin{equation}\label{eq:p_gen}
|U_0|\geq \Big(\frac{4k+2}{(k+1)^2}-k(\frac{1}{2}-\alpha)\Big)n,
\end{equation}
then $K$ contains a monochromatic $F_{n}$ with center in $V_0$.  In particular, if $|U_0|=0$, then \eqref{eq:p_gen} is equivalent to 
\begin{equation}\label{eq:al}
\alpha\leq \frac{1}{2}-\frac{4k+2}{k(k+1)^2}
\end{equation}
(which is useful provided $k\geq 3$).
\end{corollary}

\begin{proof}
Without loss of generality, say $\chi=1$ and for the purpose of this proof, say color 1 is blue and color 2 is red.  Note that $U_0$ contains a red matching of size $\floor{\frac{|U_0|}{2}}$. So if there is a red copy of $F_{n-p}$ in $\bigcup_{v\in V_0}N_B(v)$, then together with the red matching in $U_0$, we will get a red copy of $F_{n}$.  

In order to get such a red copy of $F_{n-p}$, we check that Lemma \ref{lem:cover_gen_partite} applies with $d=2.5-\alpha+\ep$, $\ell=\frac{1}{\ep}$, and $p=\floor{\frac{|U_0|}{2}}$.  Indeed, by the definition of $k$, we have $k(2n+2\ell) < (k+1)|V_0|$, which implies $|V_0| > \frac{k}{k+1}(2n+2\ell) > \frac{2k}{k+1}n$. Thus, we have
\begin{align*}
|V_0|> \frac{2k}{k+1}n
&= (1+\frac{1}{k})\Big(2n-\Big(\frac{4k+2}{(k+1)^2}-k(\frac{1}{2}-\alpha)\Big)n\Big)-(k+1)(\frac{1}{2}-\alpha)n \\
&\stackrel{\eqref{eq:p_gen}}{\geq} (1+\frac{1}{k})(2n-|U_0|)-(k+1)(\frac{1}{2}-\alpha)n \\
&\geq (1+\frac{1}{k})(2n-2\floor{\frac{|U_0|}{2}}-1)-(k+1)(\frac{1}{2}-\alpha)n \\
&= (1+\frac{1}{k})(2n-2\floor{\frac{|U_0|}{2}})-(k+1)\Big((\frac{1}{2}-\alpha)n+\frac{1}{k}\Big) \\
&\geq (1+\frac{1}{k})(2n-2\floor{\frac{|U_0|}{2}})-(k+1)\Big((\frac{1}{2}-\alpha+\ep)n-\frac{2}{\ep}\Big),
\end{align*}
where the last inequality holds since $n\geq \frac{3}{\ep^2}\geq \frac{1}{k\ep}+\frac{2}{\ep}$.  
Thus Lemma \ref{lem:cover_gen_partite} applies with $d=2.5-\alpha+\ep$, $\ell=\frac{1}{\ep}$, and $p=\floor{\frac{|U_0|}{2}}$ and we get the desired red copy of $F_{n-p}$ in $\bigcup_{v\in V_0}N_B(v)$.
\end{proof}

\subsection{Main proof}

Now we prove that for all $\ep>0$ and all $n\geq \frac{384}{\ep^2}$, $R(F_n)\leq (5+\ep)n.$

\begin{proof}[Proof of Theorem \ref{thm:main}]
Let $0<\ep'<\frac{1}{2}$ and let $n\geq \frac{384}{(\ep')^2}$.  Let $N'$ be an integer with $N'\geq (5+\ep')n$ and let $K'$ be a 2-colored complete graph on $N'$ vertices.  Now let $r$ be an integer such that $N'-\frac{\ep n}{2}+1\leq 5n+2r+1\leq N'$.  Finally set $\ep=\frac{r}{n}$ and $N=(5+2\ep)n+1$.  Note that $\frac{\ep'}{4}\leq \ep<\frac{\ep'}{2}$ and thus 
\begin{equation}
n\geq \frac{384}{(\ep')^2}\geq \frac{24}{\ep^2}.
\end{equation}
Let $V\subseteq V(K')$ with $|V|=N$ and let $K$ be the 2-colored complete graph induced by $V$.  We will show that there is a monochromatic $F_n$ in $K\subseteq K'$.

Start with a vertex $v$ of maximum monochromatic degree, and without loss of generality say $N_1:=|N_1(v)|\geq |N_2(v)|=:N_2$.  Say color 1 is blue and color 2 is red.  We may assume that for all $i\in [2]$ there is no matching of size $n$ in $G_i[N_i(v)]$ and no copy of $F_n$ in $G_{3-i}[N_i(v)]$, so by Corollary \ref{cor:LL} we may assume $3n > N_1\geq N_2 > 2n$.  Define $\alpha\geq 0$ such that $N_1=(2.5+\ep+\alpha)n$ and thus $N_2= (2.5+\ep-\alpha)n$.  So $0\leq \alpha<\frac{1}{2}-\ep$.  Note that for every color $i\in [2]$, and every vertex $v$, we have 
\begin{equation}\label{eq:maxdeg}
(2.5-\alpha+\ep)n\leq d_i(v)\leq (2.5+\alpha+\ep)n.
\end{equation}
For all $i\in [2]$, let $M_i$ be the largest matching in $G_i[N_i(v)]$ and let 
\begin{equation}\label{eq:def}
\mathrm{def}_i:=N_i-2|M_i|=N_i-|V(M_i)|>N_i-2n=\begin{cases}
(\frac{1}{2}+\alpha+\ep)n & \text{ if } \chi=1\\
(\frac{1}{2}-\alpha+\ep)n & \text{ if } \chi=2
\end{cases}.
\end{equation}
For all $\chi\in [2]$, we apply Lemma \ref{thm:GE} to get an Edmonds-Gallai partition $\{A_\chi, C_\chi, D^\chi_1, \dots, D^\chi_{p_\chi}\}$ of $G_\chi[N_\chi(v)]$.  Set $a_\chi=|A_\chi|$, $c_\chi=|C_{\chi}|$, and $d_i^\chi=|D_i^\chi|$.  By Lemma \ref{lem:structure} we have $a_\chi\leq n-1-\frac{1}{2}(c_\chi+\sum_{i\in p_\chi}(d_i^\chi-1))\leq n-1$.

For all $\chi\in [2]$, let $V_1^\chi, \dots, V_{p_1'}^\chi$ be the sets from $C_\chi, D^\chi_1, \dots, D^\chi_{p_\chi}$ which have cardinality at most $\frac{6}{\ep}$ and let $V_{p_\chi'+1},\dots, V_{p_\chi+1}$ be the sets from $C_\chi, D^\chi_1, \dots, D^\chi_{p_\chi}$ which have cardinality greater than $\frac{6}{\ep}$.  Note that the number of such sets which have cardinality greater than $\frac{6}{\ep}$ (i.e.~$p_\chi+1-p'_{\chi}$) is less than $\frac{N_\chi}{6/\ep}<\frac{3n}{6/\ep}=\frac{\ep n}{2}$.

Now let $S_\chi=\bigcup_{i\in [p'_\chi]}V_i^\chi$ and $T_\chi=\bigcup_{i\in [p'_\chi+1, p_\chi+1]}V_i^\chi$ and $s_\chi=|S_\chi|$ and $t_\chi=|T_\chi|$.  Note that $a_\chi+s_\chi+t_\chi=N_\chi$ and in particular 
\begin{equation}\label{eq:s+t}
s_\chi+t_\chi=N_\chi-a_\chi>N_\chi-n=\begin{cases}
(1.5+\alpha+\ep)n & \text{ if } \chi=1\\
(1.5-\alpha+\ep)n & \text{ if } \chi=2
\end{cases}.
\end{equation}
Also note that by Lemma \ref{lem:structure} we have $p_\chi=a_\chi+\mathrm{def}_\chi$ and thus 
\begin{equation}\label{eq:s}
s_\chi\geq p'_\chi\geq p_\chi-\frac{\ep n}{2}=a_\chi+\mathrm{def}_\chi-\frac{\ep n}{2}>\begin{cases}
a_1+(\frac{1}{2}+\alpha+\frac{\ep}{2})n & \text{ if } \chi=1\\
a_2+(\frac{1}{2}-\alpha+\frac{\ep}{2})n & \text{ if } \chi=2
\end{cases}.
\end{equation}
Finally, combining \eqref{eq:def} and \eqref{eq:s} we have
\begin{equation}\label{eq:2st}
2s_\chi+t_\chi\geq N_\chi+\mathrm{def}_\chi-\frac{\ep n}{2}> \begin{cases}
(3+2\alpha+\frac{3\ep}{2})n & \text{ if } \chi=1\\
(3-2\alpha+\frac{3\ep}{2})n & \text{ if } \chi=2
\end{cases}.
\end{equation}

\begin{claim}\label{clm:siti}
Either we have a monochromatic $F_n$ or both of the following hold:
\begin{enumerate}
\item For all $i\in [2]$, if $s_i> n+\frac{\ep}{6}$, then $s_i+t_i\leq 2n+\frac{6}{\ep}$.
\item $s_1\leq n+\frac{6}{\ep}$ or $s_1> (1+2\alpha+\frac{\ep}{2})n$.
\end{enumerate}
\end{claim}

\begin{proofclaim}
(i) If for some $i\in [2]$, $s_i> n+\frac{6}{\ep}$ and $s_i+t_i> 2n+\frac{6}{\ep}$ then we would directly have the required fan in color $3-i$ in the complete multipartite graph induced by $G_{3-i}[S_i\cup T_i]$ (by Lemma \ref{lem:partite_matching}). 

(ii) Suppose for contradiction that there is no monochromatic $F_n$ and $n+\frac{6}{\ep}< s_1\leq (1+2\alpha+\frac{\ep}{2})n$.  So by \eqref{eq:s} we have $a_1+(\frac{1}{2}+\alpha+\frac{\ep}{2})n< s_1\leq (1+2\alpha+\frac{\ep}{2})n$ and thus $a_1<(\frac{1}{2}+\alpha)n$.  So we have 
$
s_1+t_1=N_1-a_1>(2+\ep)n\geq 2n+\frac{6}{\ep},
$
contradicting (i).
\end{proofclaim}

\begin{figure}[ht]
     \centering
     \begin{subfigure}[t]{0.42\textwidth}
        \centering
        \includegraphics[width=\textwidth]{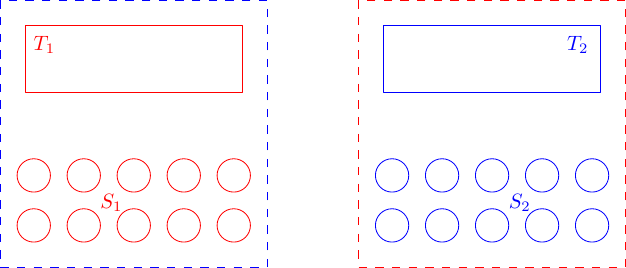}
        \caption{The monochromatic complete multipartite graphs in $G_B(v)$ and $G_R(v)$ respectively.}
        \label{fig:main1}
     \end{subfigure}
     \hfill
     \begin{subfigure}[t]{0.42\textwidth}
        \centering
         \includegraphics[width=\textwidth]{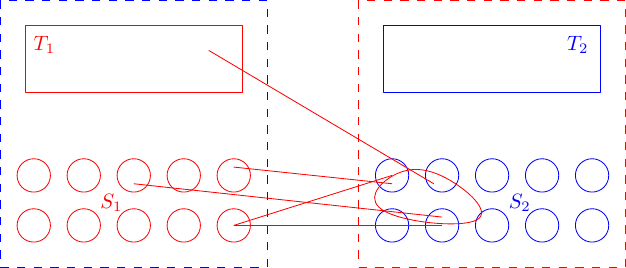}
         \caption{Using the edges between the complete multipartite graphs in an attempt to get a larger red fan.}
         \label{fig:main2}
     \end{subfigure}
     \hfill
        \caption{}
        \label{fig:main}
\end{figure}

\begin{claim}\label{clm:deg}
Either we have a monochromatic $F_n$ or all of the following hold:
\begin{enumerate}
\item For all $u\in S_1$, $d_R(u, V\setminus (S_1\cup T_1))\leq a_1+\frac{6}{\ep}\stackrel{\eqref{eq:s}}{<}s_1-(\frac{1}{2}+\alpha+\frac{\ep}{2})n$.
\item For all $v\in S_2$, $d_B(v, V\setminus (S_2\cup T_2))\leq a_2+2\alpha n+\frac{6}{\ep}\stackrel{\eqref{eq:s}}{<}s_2-(\frac{1}{2}-3\alpha+\frac{\ep}{2})n$.
\item If $s_1>n+\frac{6}{\ep}$, then for all $u\in S_1$, $d_R(u, S_2)\leq 2n-(s_1+t_1)+\frac{12}{\ep}\stackrel{\eqref{eq:2st}}{<}s_1-(1+2\alpha+\ep)n$.
\item If $s_2>n+\frac{6}{\ep}$, then for all $v\in S_2$, $d_B(v, S_1)\leq 2n-(s_2+t_2)+\frac{12}{\ep}$.
\end{enumerate}
\end{claim}

\begin{proofclaim}
(i) If there exists $u\in S_1$ with  $d_R(u, V\setminus (S_1\cup T_1))>a_1+\frac{6}{\ep}$, then $$d_R(u)> s_1+t_1-\frac{6}{\ep}+a_1+\frac{6}{\ep}=(2.5+\alpha+\ep)n,$$ a contradiction.  

(ii) If there exists $v\in S_2$ with  $d_B(v, V\setminus (S_2\cup T_2))> a_2+2\alpha n+\frac{6}{\ep}$, then $$d_B(v)> s_2+t_2-\frac{6}{\ep}+a_2+2\alpha n+\frac{6}{\ep}=(2.5+\alpha+\ep)n,$$ a contradiction.  

(iii) Suppose $s_1>n+\frac{\ep}{6}$ and there exists $u\in S_1$ with red degree greater than $2n-(s_1+t_1)+\frac{12}{\ep}$ to $S_2$.  If there is a red matching of size greater than $2n-(s_1+t_1)+\frac{12}{\ep}$ in $G_R':=G[N_R(u)\cap S_2, S_1\cup T_1]$, then by applying Lemma \ref{lem:main}(i) with $X=S_1$ and $Y=T_1$ (since $s_1> n+\frac{6}{\ep}$ and all parts in $S_1$ have size at most $\frac{6}{\ep}$) we have a red $F_n$.  Otherwise there is a vertex cover of $G_R'$ of size at most $2n-(s_1+t_1)+\frac{12}{\ep}$.  Hence there is a vertex in $S_2$ with blue degree at least $$s_2+t_2-\frac{6}{\ep}+s_1+t_1-(2n-(s_1+t_1)+\frac{12}{\ep})\stackrel{\eqref{eq:s+t}}{>}(2.5+\alpha+\ep)n,$$ a contradiction.  

(iv) Suppose $s_2>n+\frac{6}{\ep}$ and there exists $v\in S_2$ with blue degree greater than $2n-(s_2+t_2)+\frac{12}{\ep}$ to $S_1$.  If there is a blue matching of size greater than $2n-(s_2+t_2)+\frac{12}{\ep}$ in $G_B':=G_B[N_B(v)\cap S_1, S_2\cup T_2]$, then by applying Lemma \ref{lem:main}(i) with $X=S_1$ and $Y=T_1$ (since $s_1> n+\frac{6}{\ep}$ and all parts in $S_1$ have size at most $\frac{6}{\ep}$) we have a blue $F_n$.  Otherwise there is a vertex cover $Q$ of $G_B'$ of size at most $2n-(s_2+t_2)+\frac{12}{\ep}$.  Note that every edge between $(N_B(v)\cap S_1)\setminus Q$ and $(S_2\cup T_2)\setminus Q$ is red and $(N_B(v)\cap S_1)\setminus Q\neq \emptyset$.  Let $q_1=|Q\cap S_1|$ and let $q_2=|Q\cap (S_2\cup T_2)|$.  So there is a vertex $u\in N_B(v)\cap S_1$ with 
\begin{equation}\label{eq:q2}
d_R(u)\geq s_1+t_1-\frac{6}{\ep}+s_2+t_2-q_2.
\end{equation}
If $q_2<s_2+t_2-a_1-\frac{6}{\ep}$, then by \eqref{eq:q2} we have $d_R(u)>s_1+t_1+a_1=N_1$, a contradiction.  So suppose $s_2+t_2-a_1-\frac{6}{\ep}\leq q_2\leq q\leq 2n-(s_2+t_2)+\frac{12}{\ep}$ in which case \eqref{eq:q2} gives
\begin{align*}
d_R(u)&\geq s_1+t_1-\frac{6}{\ep}+s_2+t_2-(2n-(s_2+t_2)+\frac{12}{\ep})\\
&=N_1-a_1+2(N_2-a_2)-2n-\frac{18}{\ep}\\
&>(3+2\alpha+\ep)n-s_1+2(1.5-\alpha)n-2n=(4+\ep)n-s_1. 
\end{align*}
If $s_1\leq (\frac{3}{2}-\alpha+\ep)n$, then this is a contradiction; so suppose $s_1> (\frac{3}{2}-\alpha+\ep)n\geq (1+\ep)n> n+\frac{6}{\ep}$.  We will now show that $|(N_B(v)\cap S_1)\setminus Q|, |(S_2\cup T_2)\setminus Q|>(\frac{1}{2}-\alpha+\ep)n$ and since every edge between $(N_B(v)\cap S_1)\setminus Q$ and $(S_2\cup T_2)\setminus Q$ is red, this will give us a vertex $u\in S_1$ which is the center of a red $F_n$.  

Note that $$|(S_2\cup T_2)\setminus Q|=s_2+t_2-q_2\geq s_2+t_2-(2n-(s_2+t_2)+\frac{12}{\ep})\stackrel{\eqref{eq:s+t}}{>}(1-2\alpha)n> (\frac{1}{2}-\alpha+\ep)n,$$ where the last inequality holds since $0\leq \alpha<\frac{1}{2}-\ep$. Also $$|(N_B(v)\cap S_1)\setminus Q|\geq 2n-(s_2+t_2)+\frac{12}{\ep}-q_1\geq q_2\geq s_2+t_2-a_1-\frac{6}{\ep} \stackrel{\eqref{eq:s+t}}{\geq} (\frac{1}{2}-\alpha+\ep)n.$$  Thus we have a vertex in $u\in S_1$ and a red matching  of size at least $(\frac{1}{2}-\alpha+\ep)n\stackrel{\eqref{eq:s+t}}{>} 2n-(s_1+t_1)+\frac{12}{\ep}$ from $(S_2\cup T_2)\setminus Q\subseteq N_R(u)\cap (S_2\cup T_2)$ to $N_B(v)\cap S_1\subseteq S_1$ and thus we are done by Lemma \ref{lem:main}(i) (since $s_1>n+\frac{6}{\ep}$). 
\end{proofclaim}

\begin{claim}
Either we have a monochromatic copy of $F_n$ or $s_i<(\frac{3}{2}+\frac{\alpha}{3}+\alpha^2)n+\frac{12}{\ep}$ for all $i\in [2]$.
\end{claim}

\begin{proofclaim}
Without loss of generality suppose $s_1\geq (\frac{3}{2}+\frac{\alpha}{3}+\alpha^2)n+\frac{12}{\ep}$.  Let $k$ be the largest integer such that $k(2n+\frac{12}{\ep}-s_1)<s_1$.   By the definition of $k$, we have $$\frac{k}{k+1}(2n+\frac{12}{\ep})< s_1\leq \frac{k+1}{k+2}(2n+\frac{12}{\ep}),$$  
and since $s_1\geq \frac{3n}{2}+\frac{12}{\ep}$, we have $k\geq 3$.  

When $k=3$, note that $(\frac{3}{2}+\frac{\alpha}{3}+\alpha^2)n+\frac{6}{\ep}<s_1<\frac{8n}{5}+\frac{12}{\ep}$ implies that $\alpha<.2<\frac{1}{2}-\frac{4\cdot 3+2}{3(3+1)^2}$ and thus we can apply Corollary \ref{cor:main} (with $U_0=\emptyset$).  
When $k=4$, note that $(\frac{3}{2}+\frac{\alpha}{3}+\alpha^2)n+\frac{6}{\ep}<s_1<\frac{5n}{3}+\frac{12}{\ep}$ implies that $\alpha<.27<\frac{1}{2}-\frac{4\cdot 4+2}{4(4+1)^2}$ and thus we can apply Corollary \ref{cor:main} (with $U_0=\emptyset$). 

When $k\geq 5$, first note that since $s_1\geq (\frac{3}{2}+\frac{\alpha}{3}+\alpha^2)n\geq n+\frac{6}{\ep}$, we have $s_1\geq (1+2\alpha+\frac{\ep}{2})n$ by Claim \ref{clm:siti}.  Now $(1+2\alpha)n+\frac{12}{\ep}<s_1<\frac{2(k+1)}{k+2}n+\frac{12}{\ep}$ implies $\alpha<\frac{1}{2}-\frac{1}{k+2}<\frac{1}{2}-\frac{4k+2}{k(k+1)^2}$ and thus we can apply Corollary \ref{cor:main} (with $U_0=\emptyset$).
\end{proofclaim}

Having established these general claims, we now split into cases depending on the sizes of $s_1$ and $s_2$.

\noindent
\textbf{Case 1.1} ($(1+2\alpha+\frac{\ep}{2})n < s_1 < (\frac{3}{2}+\frac{\alpha}{3}+\alpha^2)n+\frac{12}{\ep}$ and $n+\frac{6}{\ep} < s_2 < (\frac{3}{2}+\frac{\alpha}{3}+\alpha^2)n+\frac{12}{\ep}$)

Let $f(s_1, s_2) = s_1s_2 - s_1(s_1-(1+2\alpha)n) - s_2(s_2-(1-2\alpha)n)$. We first  show that $f(s_1, s_2) > 0$ on the domain implicitly defined in this case. 

We have
\begin{align*}
f(s_1, s_2) &= s_1s_2 - s_1(s_1-(1+2\alpha)n) - s_2(s_2-(1-2\alpha)n)\\
&=(2n - s_1)\left(s_1 - (1+2\alpha)n\right) + (2n - s_2)\left(s_2 - (1-2\alpha)n\right) + (2n - s_1)(2n - s_2)>0
\end{align*}
where the last inequality holds by the bounds on $s_1$ and $s_2$ in this case.  

On the other hand, by Claim \ref{clm:deg}(iii) and (iv), we have 
$$s_1s_2 = \sum_{u\in S_1}d_2(u, S_2) + \sum_{v\in S_2}d_1(v, S_1) < s_1(s_1-(1+2\alpha)n) + s_2(s_2-(1-2\alpha)n)$$ 
which contradicts $f(s_1, s_2) > 0$.

\noindent
\textbf{Case 1.2} ($(1+2\alpha+\frac{\ep}{2})n < s_1 < (\frac{3}{2}+\frac{\alpha}{3}+\alpha^2)n+\frac{12}{\ep}$ and $(\frac{1}{2}-\alpha+\frac{\ep}{2})n+a_2 \leq s_2 \leq n+\frac{6}{\ep}$)

First note that we can relax the bounds to $(1+2\alpha)n \leq s_1 \leq (\frac{3}{2}+\frac{\alpha}{3}+\alpha^2+\ep)n$ and $(\frac{1}{2}-\alpha)n \leq s_2 \leq (1+\ep)n$.  

Let $f(s_1, s_2) = s_1s_2 - s_1(s_1-(1+2\alpha+\ep)n) - s_2(s_2-(\frac{1}{2}-3\alpha+\frac{\ep}{2})n)$.  We first show that $f(s_1, s_2) > 0$ on the domain $[(1+2\alpha)n, (\frac{3}{2}+\frac{\alpha}{3}+\alpha^2+\ep)n]\times [(\frac{1}{2}-\alpha)n, (1+\ep)n]$. 

Note that
$$
\begin{vmatrix}
\frac{\partial^2 f}{\partial s_1^2} & \frac{\partial^2 f}{\partial s_2\partial s_1} \\
\frac{\partial^2 f}{\partial s_1\partial s_2} & \frac{\partial^2 f}{\partial s_2^2}
\end{vmatrix} 
=
\begin{vmatrix}
-2 & 1 \\
1 & -2
\end{vmatrix}
=3>0,
$$
and $\frac{\partial^2 f}{\partial s_1^2} = -2 < 0$, thus $f$ is strictly concave.  So the minimum of $f$ on this closed rectangular domain must occur at one of the four corners. Evaluating $f$ at these four corners, we have:
\begin{align*}
f((1+2\alpha)n, (\frac{1}{2}-\alpha)n) &= ( \frac{1}{2} - \alpha + (\frac{5}{4} + \frac{3}{2}\alpha)\ep )n^2 > 0 \\
f((1+2\alpha)n, (1+\ep)n) &= ( \frac{1}{2} - \alpha + (1 + \alpha - \frac{1}{2}\ep)\ep )n^2 > 0 \\
f((\frac{3}{2}+\frac{\alpha}{3}+\alpha^2+\ep)n, (1+\ep)n) &= ( \frac{1}{4} - \frac{1}{3}\alpha - (\alpha - \frac{2}{3})^2\alpha^2 + (1 - \alpha + \frac{1}{2}\ep)\ep )n^2 > 0 \\
f((\frac{3}{2}+\frac{\alpha}{3}+\alpha^2+\ep)n, (\frac{1}{2}-\alpha)n) &= ( (\frac{3}{4} - (\alpha - \frac{1}{6})^2)\alpha^2 + (\frac{1}{4} + \frac{1}{6}\alpha - \alpha^2)\ep )n^2 > 0
\end{align*}
Where all the inequalities hold since $0\leq \alpha < \frac{1}{2}-\ep$.  Thus $f(s_1, s_2) > 0$ on the entire domain. 

However, by Claim \ref{clm:deg}(ii) and (iii), we have 
$$s_1s_2 = \sum_{u\in S_1}d_2(u, S_2) + \sum_{v\in S_2}d_1(v, S_1) < s_1(s_1-(1+2\alpha+\ep)n) + s_2(s_2-(\frac{1}{2}-3\alpha+\frac{\ep}{2})n)$$ 
which contradicts $f(s_1, s_2) > 0$.

\noindent
\textbf{Case 2.1} ($(\frac{1}{2}+\alpha+\frac{\ep}{2})n+a_1 < s_1 \leq n+\frac{6}{\ep}$ and $n+\frac{6}{\ep} < s_2 < (\frac{3}{2}+\frac{\alpha}{3}+\alpha^2)n+\frac{12}{\ep}$)

Let $f(s_1, s_2) = s_1s_2 - s_1(s_1-(\frac{1}{2}+\alpha)n) - s_2((2+\ep)n-(s_2+t_2))$.  We first show that $f(s_1, s_2) > 0$. 

First note that by the case we have $(\frac{1}{2}+\alpha)n\leq s_1\leq n+\frac{6}{\ep}< s_2$ and $(1.5-\alpha+\ep)n-t_2\stackrel{\eqref{eq:s+t}}{\leq} s_2\leq (2+\ep)n-t_2$.  Using $s_2 + t_2 \ge (1.5-\alpha+\ep)n$, we have
\begin{align*}
f(s_1, s_2) &= s_1s_2 - s_1(s_1-(\frac{1}{2}+\alpha)n) - s_2((2+\ep)n-(s_2+t_2))\\
&= s_1s_2 - s_1^2 + s_1(\frac{1}{2}+\alpha)n + s_2(s_2+t_2) - s_2(2+\ep)n\\
&\geq s_1s_2 - s_1^2 + s_1(\frac{1}{2}+\alpha)n + s_2(1.5-\alpha+\ep)n - s_2(2+\ep)n\\
&= s_1(s_2 - s_1) + s_1(\frac{1}{2}+\alpha)n - s_2(\frac{1}{2}+\alpha)n\\
&= (s_2 - s_1)(s_1 - (\frac{1}{2}+\alpha)n)>0.
\end{align*}
Where the last inequality holds by the bounds on $s_1$ and $s_2$ in this case (in particular $s_2 > s_1$). 

On the other hand, by Claim \ref{clm:deg}(i) and (iv), we have 
$$s_1s_2 = \sum_{u\in S_1}d_2(u, S_2) + \sum_{v\in S_2}d_1(v, S_1) < s_1(s_1-(\frac{1}{2}+\alpha)n) + s_2((2+\ep)n-(s_2+t_2)),$$ 
which contradicts $f(s_1, s_2) > 0$.

\noindent
\tbf{Case 2.2} ($\frac{1}{2}+\alpha+\frac{\ep}{2})n+a_1< s_1\leq n+\frac{6}{\ep}$ and $(\frac{1}{2}-\alpha+\frac{\ep}{2})n+a_2\leq s_2\leq  n+\frac{6}{\ep}$)

First note that in this case, because of the bounds on $a_1$ and $a_2$, \eqref{eq:s+t} becomes
\begin{equation}\label{eq:s+t'}
s_\chi+t_\chi=N_\chi-a_\chi>\begin{cases}
(2+2\alpha+\ep)n & \text{ if } \chi=1\\
(2-2\alpha+\ep)n & \text{ if } \chi=2
\end{cases}.
\end{equation}

We begin with some expository remarks to explain why this case is different from the others.  In this case, we have that both $s_1$ and $s_2$ are too small to apply Claim \ref{clm:deg}(iii) or Claim \ref{clm:deg}(iv).  However, we have that because $s_1$ and $s_2$ are small, $t_1$ and $t_2$ are large.  In particular, by \eqref{eq:s+t'}, $s_1+t_1$ is large enough that if there were a sufficiently large red matching inside $T_1$, we would have the desired red copy of $F_n$ with center in $S_1$.  If there is not a sufficiently large red matching inside $T_1$, then $T_1$ contains a large blue clique.  Importantly, every edge from this large blue clique in $T_1$ sends red edges to all of $S_1$.  We will be able to make use of Corollary \ref{cor:main} in this case unless the blue clique in $T_1$ is too small.  When that happens we must make some ad-hoc arguments to finish the proof.  

Let $M$ be a maximum red matching inside $T_1$ and let $m=|M|$.  First note that since $s_1+t_1-\frac{6}{\ep}=(2.5+\alpha+\ep)n-a_1-\frac{6}{\ep}> (2+2\alpha)n$, we have
\begin{equation}\label{eq:t1-2m}
t_1-2m> s_1-\frac{6}{\ep}
\end{equation}
as otherwise if $t_1-2m\leq s_1-\frac{6}{\ep}$ we have a red fan (which covers all of $T_1$) on $$s_1-\frac{6}{\ep}+2m+t_1-2m=s_1+t_1-\frac{6}{\ep}=(2.5+\alpha+\ep)n-a_1-\frac{6}{\ep}>(2+2\alpha)n$$ vertices and we are done.  Next note that 
\begin{equation}\label{eq:s+m}
m< n-s_1+\frac{6}{\ep} 
\end{equation}
as otherwise if $m\geq n-s_1+\frac{6}{\ep}$, then by \eqref{eq:t1-2m}, we have $t_1-2m>s_1-\frac{6}{\ep}$ and thus for all $u\in S_1$ there is a red fan with $u$ as the center using $M$ and a red matching of size at least $s_1-\frac{6}{\ep}$ from $S_1$ to $T_1'$ which has more than $2m+2s_1-\frac{12}{\ep}\geq 2(n-s_1+\frac{6}{\ep})+2s_1-\frac{12}{\ep}=2n$ vertices and we are done.  Finally, we have
\begin{equation}\label{eq:t2m}
t_1-2m\stackrel{\eqref{eq:2st}}{>}(3+2\alpha+\frac{3\ep}{2})n-2(s_1+m)\stackrel{\eqref{eq:s+m}}{>}(1+2\alpha+\ep)n.
\end{equation}

Let $T_1'=T_1\setminus V(M)$.  Note that by the maximality of $M$, $T_1'$ induces a blue clique of order $t_1-2m$ and every edge between $T_1'$ and $S_1$ is red.  Before dealing with the inequalities bounding the number of edges between $S_1$ and $S_2$, we first apply Corollary \ref{cor:main} with $V_0=T_1'$ and $U_0=S_1$ to rule out a number of cases.  

Let $k$ be the largest integer such that $k(2n-(t_1-2m))<t_1-2m$ and note that the definition of $k$ implies 
\begin{equation}\label{eq:t2m2}
\frac{2k}{k+1}n< t_1-2m\leq \frac{2(k+1)}{k+2}n.
\end{equation}
Combining \eqref{eq:t2m} with \eqref{eq:t2m2} gives 
\begin{equation}\label{eq:1/2-}
\alpha<\frac{1}{2}-\frac{1}{k+2}-\frac{\ep}{2}.
\end{equation}

If $k\geq 2$, then since 
$$|U_0|=|S_1|\geq (\frac{1}{2}+\alpha+\frac{\ep}{2})n\stackrel{\eqref{eq:1/2-}}{>} (\frac{4k+2}{(k+1)^2}-k(\frac{1}{2}-\alpha))n$$ we can apply Corollary \ref{cor:main} with $V_0=T_1'$ and $S_1=U_0$ to get the monochromatic $F_n$ with center in $V_0$.

So we are done unless $k=1$.  In this case, by \eqref{eq:1/2-}, we have $$0\leq \alpha<\frac{1}{6}-\frac{\ep}{2}$$ and also $$(3+2\alpha+\frac{3\ep}{2})n-2s_1-2m\stackrel{\eqref{eq:t2m}}{<}t_1-2m\stackrel{\eqref{eq:t2m2}}{\leq }\frac{4}{3}n$$ which implies
\begin{equation}\label{eq:s1m}
(\frac{5}{6}+\alpha+\ep)n< s_1+m \stackrel{\eqref{eq:s+m}}{<} n+\frac{6}{\ep}.
\end{equation}

The following claim adds to Claim \ref{clm:deg} using the specific circumstances of this remaining case.  

\begin{claim}\label{clm:t2}~
\begin{enumerate}
\item If $t_2\geq n$, then every vertex in $S_2$ has blue degree at most $n-s_2+\frac{6}{\ep}$ to $S_1$. 
\item If $t_2\leq n$, then every vertex in $S_2$ has blue degree at most $2n-(s_2+t_2)+\frac{6}{\ep}\stackrel{\eqref{eq:s+t'}}{<}2\alpha n$ to $S_1$.
\item Every vertex in $S_1$ has red degree at most $n-s_1-m+\frac{6}{\ep}$ to $S_2$.
\end{enumerate}
\end{claim}

\begin{proofclaim}
(i) Suppose $t_2\geq n$ and there exists $v\in S_2$ with blue degree greater than $n-s_2+\frac{6}{\ep}$ to $S_1$.  If there exists a blue matching of size greater than $n-s_2+\frac{6}{\ep}$ in $G_B':=G_B[N_B(v)\cap S_1, T_2]$, then by Lemma \ref{lem:main}(iii) we have a blue $F_n$. Otherwise there is a vertex cover of $G_B'$ of size at most $n-s_2+\frac{6}{\ep}$ and thus there is a vertex in $S_1$ with red degree at least $$s_1+t_1-\frac{6}{\ep}+t_2-(n-s_2+\frac{6}{\ep})\stackrel{\eqref{eq:s+t'}}{>} (3+2\ep)n-\frac{12}{\ep}>3n,$$ a contradiction.

(ii) Suppose $t_2< n$ and there exists a $v\in S_2$ with blue degree greater than $2n-(s_2+t_2)+\frac{6}{\ep}$ to $S_1$.  If there exists a blue matching of size greater than $2n-(s_2+t_2)+\frac{6}{\ep}$ in $G_B':=G_B[N_B(v)\cap S_1, T_2]$, then by Lemma \ref{lem:main}(ii) we have a blue $F_n$.  Otherwise there is a vertex cover of $G_B'$ of size at most $2n-(s_2+t_2)+\frac{6}{\ep}$ and thus there is a vertex in $S_1$ with red degree greater than 
$$s_1+t_1-\frac{6}{\ep}+t_2-(2n-(s_2+t_2)+\frac{6}{\ep})\stackrel{\eqref{eq:s+t'}}{>} (2+2\ep)n-\frac{12}{\ep}+t_2>(2.5+\alpha+\ep)n,$$
a contradiction (where we used the fact that by the case we have $t_2=N_2-(s_2+a_2)> (1-2\alpha+\frac{\ep}{2})n>(\frac{1}{2}+\alpha+\ep)n$ which holds since $\alpha< \frac{1}{6}-\frac{\ep}{2}$ by \eqref{eq:s1m}).

(iii) First note that by \eqref{eq:t2m}, we have $|T_1'|=t_1-2m>(1+2\alpha+\ep)n> n.$  Now suppose there exists $v\in S_1$ with red degree greater than $n-s_1-m+\frac{6}{\ep}$ to $S_2$.  If there exists a red matching of size greater than $n-s_1-m+\frac{6}{\ep}$ in $G_R':=G_R[N_R(v)\cap S_2, T_1']$, then by Lemma \ref{lem:main}(iii) we have a red $F_n$.  Otherwise there is a vertex cover of $G_R'$ of size at most $n-s_1-m+\frac{6}{\ep}$ and thus there exists a vertex in $S_2$ with blue degree at least 
\begin{align*}
s_2+t_2-\frac{6}{\ep}+t_1-2m-(n-s_1-m+\frac{6}{\ep})&=s_1+s_2+t_1+t_2-n-m-\frac{12}{\ep}\\
&\stackrel{\eqref{eq:s+m}}{>} 2s_1+s_2+t_1+t_2-2n-\frac{18}{\ep}\\
&\stackrel{\eqref{eq:2st},\eqref{eq:s+t'}}{>}(3+2\alpha)n+(2-2\alpha)n-2n=3n
\end{align*}
a contradiction.
\end{proofclaim}

\noindent
\textbf{Subcase 2.2(a)} ($(1-2\alpha)n+\frac{6}{\ep} < s_2 \leq n+\frac{6}{\ep}$)

By the case we have $(\frac{1}{2}+\alpha)n \leq s_1 \leq (1+\ep)n$ and $(1-2\alpha)n < s_2 \leq (1+\ep)n$, with the additional constraint $\alpha < \frac{1}{6}-\ep$. 

Let $f(s_1, s_2) = s_1s_2 - s_1(s_1-(\frac{1}{2}+\alpha)n) - s_2(2\alpha n)$. We first show that $f(s_1, s_2) > 0$ on the domain $\left[(\frac{1}{2}+\alpha)n, (1+\ep)n\right]\times \left[(1-2\alpha)n, (1+\ep)n\right]$.

Note that 
\begin{align*}
f(s_1, s_2) &= s_2(s_1 - 2\alpha n) - s_1(s_1-(\frac{1}{2}+\alpha)n)\\
&> (1-2\alpha)n(s_1 - 2\alpha n) - s_1(s_1-(\frac{1}{2}+\alpha)n) =: g(s_1)
\end{align*}
where the strict inequality holds since $s_1 - 2\alpha n \geq (\frac{1}{2}-\alpha)n > 0$ and $s_2 > (1-2\alpha)n$.  Since $g(s_1)$ is a concave quadratic in $s_1$, its minimum on the interval $\left[(\frac{1}{2}+\alpha)n, (1+\ep)n\right]$ must occur at one of the endpoints.  Evaluating $g$ at the endpoints gives
\begin{align*}
g((\frac{1}{2}+\alpha)n) &= (1-2\alpha)n(\frac{1}{2}-\alpha)n = 2(\frac{1}{2}-\alpha)^2 n^2 > 0 \\[2ex]
g((1+\ep)n) &= (1-2\alpha)n(1+\ep-2\alpha)n - (1+\ep)n(\frac{1}{2}-\alpha+\ep)n \\
&= ( 3(\frac{1}{6} - \alpha - \ep) + 4\alpha^2 + \ep(\frac{5}{2} - \alpha - \ep) )n^2 > 0
\end{align*}
and thus $f(s_1, s_2) > g(s_1) \geq 0$ on the entire domain.

Note that $2\alpha n\stackrel{\eqref{eq:s+t'}}{>}2n-(s_2+t_2)+\frac{6}{\ep}$ and since $s_2>(1-2\alpha)n+\frac{6}{\ep}$, we have $2\alpha n>n-s_2+\frac{6}{\ep}$. So by Claim \ref{clm:t2}(i) or (ii) (depending on the size of $t_2$), we have 
$$s_1s_2 = \sum_{u\in S_1}d_2(u, S_2)+\sum_{v\in S_2}d_1(v, S_1) < s_1(s_1-(\frac{1}{2}+\alpha)n)+s_2\cdot 2\alpha n$$
which contradicts $f(s_1, s_2) > 0$.

\noindent
\textbf{Subcase 2.2(b)} ($(\frac{1}{2}-\alpha+\frac{\ep}{2})n < s_2 \leq (1-2\alpha)n+\frac{6}{\ep}$)

First note that by the case we have $(\frac{5}{6}+\alpha)n-m \leq s_1 \leq (1+\frac{\ep}{2})n-m$ and $(\frac{1}{2}-\alpha)n \leq s_2 \leq (1-2\alpha+\ep)n$.  Also note that since $s_1>(\frac{1}{2}+\alpha)n$, we have $m\leq (1+\frac{\ep}{2})n-s_1<(\frac{1}{2}-\alpha)n$.

Define $f(s_1,s_2)=s_1s_2 - s_1(n-s_1-m)-s_2(n-s_2)$ and $g(s_1, s_2)=s_1s_2 - s_1(s_1-(\frac{1}{2}+\alpha)n)-s_2(s_2-(\frac{1}{2}-3\alpha)n)$.  We first show that for all $(s_1, s_2)$ in the domain implicitly defined by this case, we have $f(s_1,s_2)+g(s_1,s_2)>0$ which implies that either $f(s_1, s_2)>0$ or $g(s_1, s_2)>0$.

Indeed, we have 
\begin{align*}
f(s_1, s_2)+g(s_1, s_2)&=2s_1s_2 - s_1((\frac{1}{2}-\alpha)n-m)  - s_2(\frac{1}{2}+3\alpha)n\\
&=s_1(s_2 - (\frac{1}{2}-\alpha)n+m) + s_2(s_1-(\frac{1}{2}+3\alpha)n),
\end{align*}
so if $s_1\geq (\frac{1}{2}+3\alpha)n$, then we are done.  Otherwise we have $\max\{(\frac{1}{2}+\alpha)n, (\frac{5}{6}+\alpha)n-m\}\leq s_1<(\frac{1}{2}+3\alpha)n$ which in particular implies that $m>(\frac{1}{3}-2\alpha)n$.  Thus, continuing from above and using the lower bounds on $m$, $s_1$, and $s_2$, we have
\begin{align*}
f(s_1, s_2)+g(s_1, s_2)&=s_1(s_2 - (\frac{1}{2}-\alpha)n+m) + s_2(s_1-(\frac{1}{2}+3\alpha)n)\\
&>(\frac{1}{2}+\alpha)n(s_2-(\frac{1}{6}-\alpha)n)-s_2\cdot 2\alpha n\\
&=s_2(\frac{1}{2}-\alpha)n-(\frac{1}{2}+\alpha)n(\frac{1}{6}-\alpha)n\\
&>(\frac{1}{2}-\alpha)n(\frac{1}{2}-\alpha)n-(\frac{1}{2}+\alpha)n(\frac{1}{6}-\alpha)n\\
&=(2(\frac{1}{6}-\alpha)^2+\frac{1}{9})n^2>0
\end{align*}
where the inequality holds because

Note that by the upper bound on $s_2$ we have $t_2 \stackrel{\eqref{eq:2st}}{>} (3-2\alpha+\frac{3}{2}\ep)-2s_2 > (1+2\alpha+\ep)n$ and thus by Claim \ref{clm:t2}(i) every vertex in $S_2$ has blue degree at most $n-s_2+\frac{12}{\ep}$ to $S_1$.
So by Claim \ref{clm:deg}(i) and (ii) and Claim \ref{clm:t2}(i) and (iii), we have 
$$s_1s_2 = \sum_{u\in S_1}d_2(u, S_2)+\sum_{v\in S_2}d_1(v, S_1) < s_1(s_1-(\frac{1}{2}+\alpha)n)+s_2(s_2-(\frac{1}{2}-3\alpha)n)$$ and
$$s_1s_2 = \sum_{u\in S_1}d_2(u, S_2)+\sum_{v\in S_2}d_1(v, S_1) < s_1(n-s_1-m)+s_2(n-s_2)$$ 
which contradicts the fact that $f(s_1,s_2)>0$ or $g(s_1, s_2)>0$.
\end{proof}

\section{Conclusion and open problems}\label{sec:conclusion}

Of course the main open problem is to further improve the bounds on $R(F_n)$.  Improving the lower bound on $R(F_n)$ by a significant amount would require proving that $R(F_n)> R(K_{1,2n}, F_n)$; i.e.~there is a lower bound on $R(F_n)$ which makes use of the fact that there is no $F_n$ in either color.  On the other hand, perhaps $R(F_n)$ is (asymptotically) equal to $R(K_{1,2n}, F_n)=(3+\sqrt{3})n+\Theta(1)$?

There is at least one piece of evidence to suggest that $R(K_{1,2n}, F_n)=R(F_n)$ is possible.  When $m=\Omega(n^2)$, Zhang, Broersma, and Chen proved that $R(K_{1,2m}, F_n)=R(F_m, F_n)=4m+1$.  Indeed, in \cite{ZBC1}, they proved that for $m\geq \max\{n^2-\frac{n}{2}, \frac{11n}{2}-4\}$, $R(F_m, F_n)=4m+1$, and in \cite{ZBC3} they proved that for $m\geq \max\{n(n-1), 6n-7\}$, $R(K_{1,2m}, F_n)=4m+1$.  

It would be interesting to see if $R(K_{1,2m}, F_n)= R(F_m, F_n)$ for other values of $m=o(n^2)$.  In general, for the off diagonal case, $R(F_m, F_n)$, Lin and Li proved the following upper bound.

\begin{theorem}[Lin, Li]\label{thm:LL}
For $m\geq n\geq 2$, $R(F_m, F_n)\leq 4m+2n$.
\end{theorem}

Note that in the case when $n\leq m<n^2-\frac{n}{2}$, our lower bound from Theorem \ref{thm:fanstar} applies and we have
$$3m+\sqrt{m^2+2n^2}-4\leq R(K_{1,2m}, F_n)\leq R(F_m, F_n)\leq 4m+2n.$$

Our proof technique for Theorem \ref{thm:main} can almost certainly be generalized to improve Theorem \ref{thm:LL}, but it is unclear what type of bound would result.  We do not pursue this here, as introducing another parameter would certainly increase the number of cases (of which there are already plenty) and thus obscure the main ideas of the paper. 

\bigskip

\noindent
\tbf{Acknowledgements and AI tool disclosure:} Thank you to J\'anos Bar\'at and Ryan Martin for the discussions which led to Problem \ref{prob:turan}.  Thank you to Deepak Bal for a helpful conversation which inspired us to improve the presentation of a number of inequalities in the paper.

This paper was entirely human-written.  Before submitting the paper, we used Gemini to essentially act as a referee.  This process highlighted a number of typos and omitted details which we were then able to address before submission.  The humans who wrote this paper take full responsibility for any remaining errors.

\bibliographystyle{abbrv}
\bibliography{references}

\newpage

\section{Appendix: A lower bound on the Ramsey number of $K_{1,2n}$ vs.~$F_n$}

\begin{example}[Specific example]
For all integers $n\geq 2$,
$$
R(K_{1,2n}, F_n)\geq 2\floor{\sqrt{3}n}+2\floor{\frac{3-\sqrt{3}}{2}n}-4>(3+\sqrt{3})n-8.
$$
\end{example}

\begin{proof}
Let $n$ be a positive integer.  Take disjoint sets $X_1, X_2, Y_1, Y_2$ such that $|X_1|=|X_2|=a:=\floor{\sqrt{3}n}-1$ and $|Y_1|=|Y_2|=b=\floor{\frac{3-\sqrt{3}}{2}n}-1$ and set $N=2a+2b$.  

Now add all red edges between $X_1$ and $X_2$, between $X_1$ and $Y_2$, and between $X_2$ and $Y_1$.   For all $i\in [2]$, between $X_i$ and $Y_i$ we put red edges so that the induced bipartite graph between $X_i$ and $Y_i$ has the property that every vertex in $X_i$ has red degree between $n-4-b$ and $n-1-b$ to $Y_i$, and every vertex in $Y_i$ has red degree between $n-4$ and $n-1$ to $X_i$. 

\begin{figure}[ht]
    \centering
    \includegraphics[scale=1]{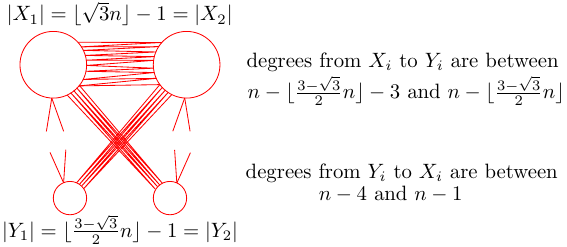}
    \caption{A graph on $N$ vertices with minimum degree at least $N-2n$ having no copy of $F_n$.} 
    \label{fig:example_special}
\end{figure}

Since $\sqrt{3}$ is irrational, we have $a+2b<\sqrt{3}n-1+(3-\sqrt{3})n-2=3n-3$ and since $a+2b$ is an integer we have $a+2b\leq 3n-4$.  This implies that the minimum degree of the red graph is at least $$a+n-4\geq 2a+2b-2n=N-2n$$ and thus there is no blue copy of $K_{1,2n}$.  Also note that the red graph is tripartite with parts $X_1, X_2, Y_1\cup Y_2$, and every vertex has degree at most $n-1$ to one of the other parts, so there is no red copy of $F_n$.  

All that remains is to show that such a bipartite graph exists between $X_i$ and $Y_i$.  By Lemma \ref{lem:GR} (with $c=n-1-b$, $d=n-1$, and $\sigma=3$), this amounts to checking that $$-3 b\leq a(n-1-b)-b(n-1)\leq 3 a.$$
Indeed, we have 
\begin{align*}
3a-(a(n-1-b)-b(n-1))&= b(n-1)-a(n-4-b)\\
&\geq (\frac{3-\sqrt{3}}{2}n-2)(n-1)-(\sqrt{3}n-1)(n-4-(\frac{3-\sqrt{3}}{2}n-2))\\
&=(3\sqrt{3}-4)n>0.
\end{align*}
Also we have 
\begin{align*}
3 b-(b(n-1)-a(n-1-b))&= a(n-1-b)-b(n-4)\\
&\geq (\sqrt{3}n-2)(n-1-(\frac{3-\sqrt{3}}{2}n-1))-(\frac{3-\sqrt{3}}{2}n-1)(n-4)\\
&=(8-3\sqrt{3})n-4 > 0 
\end{align*}
where the last inequality holds since $n\geq 2$.  
\end{proof}

\end{document}